\documentclass[a4paper, 11pt]{amsart} 

\usepackage{xcolor}
\usepackage{amsmath}
\usepackage{amssymb}			
\usepackage{amsfonts}
\usepackage{amsthm}
\usepackage{mathtools}
\usepackage{enumitem}

\usepackage{multirow}
\usepackage{mathabx,epsfig}

\usepackage{verbatim}

\usepackage[colorlinks=true, urlcolor=blue, linkcolor=red]{hyperref}
\usepackage{tikz-cd}

\usepackage{tikz}
  \usetikzlibrary{matrix}
  \usetikzlibrary{fit}
  \usetikzlibrary{backgrounds}
  \usetikzlibrary{arrows}
  \usetikzlibrary{shapes}

\def\QQ{{\mathbb Q}}
\def\AA{{\mathbb A}}
\def\CC{{\mathbb C}}
\def\RR{{\mathbb R}}

\def\ZZ{{\mathbb Z}}
\def\FF{{\mathbb F}}

\def\GG{{\mathbb G}}

\def\fp{{\mathfrak p}}

\def\GL{{\mathrm{GL}}}
\def\SL{{\mathrm{SL}}}

\def\Fi{{F^\mathrm{inn}}}

\def\hG{\widehat{G}}
\def\hGd{\widehat{G}^{\mathrm{der}}}
\def\hZ{\widehat{Z}}

\title[Irreducibility and Monodromy of Galois Representations]{Irreducibility and Monodromy of Automorphic Galois Representations of $\GL(4)$} 
\author{Alireza Shavali} 

\theoremstyle{plain}
\newtheorem{theo}{Theorem}[section]
\newtheorem{prop}[theo]{Proposition}
\newtheorem{cor}[theo]{Corollary}
\newtheorem{lemma}[theo]{Lemma}

\theoremstyle{remark}

\newtheorem{rem}[theo]{Remark}

\theoremstyle{definition}
\newtheorem{defi}[theo]{Definition}
\newtheorem{exmp}[theo]{Example}

\begin{document} 
\maketitle

\begin{abstract}
    We prove that over totally real fields, the $p$-adic Galois representations attached to non-self-dual regular algebraic cuspidal automorphic representations of $\GL(4)$ are irreducible. We then develop the theory of extra-twists in a general setting and use it to compute the monodromy group (over $\QQ$) of these Galois representations, in both self-dual and non-self-dual settings, and prove $p$-adic and residual big image results.
\end{abstract}

{\small \tableofcontents}
\renewcommand{\baselinestretch}{1}\normalsize

\section{Introduction}

For a number field $K$, the Langlands program predicts a connection between the representation theory of the locally compact group $\GL_n(\AA_K)$ and representation theory of the profinite group $\Gamma_K=\mathrm{Gal}(\overline{K}/K)$. More precisely, one expects a correspondence between algebraic isobaric automorphic representations of $\GL_n(\AA_K)$ and $n$-dimensional continuous semisimple representations of $\Gamma_K$. One expects to be able to translate different features of the objects from one world to the other. This translation is often governed by the functoriality principles introduced by Langlands. Let us describe two instances of this philosophy that we are going to be concerned with in this paper.

On the automorphic side, one has a way of constructing automorphic representations of $\GL_{n}$ from automorphic representations of smaller groups by means of parabolic induction. For a partition $n=n_1+\cdots+n_k$ into positive integers and automorphic representations $\pi_i$ of $\GL_{n_i}$, one constructs the isobaric sum $\pi=\pi_1 \boxplus \cdots \boxplus \pi_k$ as a subquotient of the parabolic induction. Cuspidal automorphic representations are essentially the ones that cannot be found in parabolic inductions. Therefore, they are the building blocks of the isobaric decomposition. On the Galois side, one has the familiar notion of irreducible representations which serve as building blocks of the direct sum decomposition of semisimple representations. Hence, one naturally expects the notion of cuspidality on the automorphic side to translate to irreducibility on the Galois side.

For the second example, we consider the notion of monodromy. Let 
$$\rho:\Gamma_K \rightarrow \GL_n(\QQ_p)\subset \GL_n(\overline{\QQ_p})$$
be a Galois representation. The $\overline{\QQ_p}$-monodromy group of $\rho$ is the Zariski closure $\mathcal{G}$ of $\rho(\Gamma_K)$ in $\GL_{n,\overline{\QQ_p}}$ and the (arithmetic) monodromy group of $\rho$ is the Zariski closure $G$ of $\rho(\Gamma_K)$ in $\GL_{n,\QQ_p}$. If $\pi$ comes from a smaller subgroup of $\GL_n$ via Langlands functoriality, then one expects the Galois representation to factor through the Langlands dual of that group. Therefore, one expects to be able to read $\mathcal{G}$ from the functoriality behavior of $\pi$. Predicting $G$ from the automorphic data is more tricky. In the case of modular forms of weight 2, Ribet first realized that this is possible using the so-called inner-twists of the form \cite{ribet1980twists}. 
\smallskip\\
\noindent\textbf{Main results.} Let us explain what is already known and what we prove in this paper. Let $K$ be a totally real field and $\pi$ be a regular algebraic cuspidal automorphic representation of $\GL_n(\AA_K)$. Then $E=\QQ(\pi)$ is known to be a number field. It is known by the work of Harris-Lan-Taylor-Thorne \cite{harris2016rigid} or Scholze \cite{scholze2015torsion}, that there is a compatible family
$
\{\rho_{\pi,\fp}:\Gamma_K \rightarrow \GL_n(\overline{E_\fp})\}_{\fp}
$
of $E$-rational $p$-adic Galois representations associated with $\pi$. Since $\pi$ is cuspidal, one expects these Galois representations to be irreducible. We refer to this prediction as the irreducibility conjecture. If $\pi$ is essentially self-dual, there has been a lot of work in the direction of this conjecture, see \cite{ramakrishnan2013decomposition}, \cite{calegari2013irreducibility}, \cite{xia2019irreducibility}, and \cite{hui2023monodromy} for small $n$, and see \cite{barnet2014potential}, \cite{patrikis2015automorphy}, and \cite{feng2025irreducibility} for higher dimensional cases. We refer to Table 1 in \cite{feng2025irreducibility} for a summary of the known results.  
If $\pi$ is not essentially self-dual, the main issue, at least in the low dimensional cases $n=3$ or $4$, is that the local-global compatibility at $\ell=p$ for these Galois representations is only known under assumptions that are stronger than irreducibility (see \cite[Theorem 1.0.6]{a2024rigidity}). This prevents one from applying potential modularity arguments to the direct summands of $\rho_{\pi,\fp}$. In fact, conditional on the local-global compatibility at $\ell=p$ (which one expects to be established soon), Calegari and Gee show the irreducibility of $\rho_{\pi,\fp}$ in the non-essentially-self-dual case for $n=3$ and $4$ \cite[\S 8]{calegari2013irreducibility} (they implicitly use local-global compatibility in \cite[Lemma 8.5]{calegari2013irreducibility}). Note that there is an error in that paper, but we believe that it does not affect the proof in the non-essentially-self-dual case. 

Böckle and Hui generalized some results of Serre on abelian Galois representations and used it to prove the irreducibility in the non-essentially-self-dual case for $n=3$ \cite{bockle2025weak}. Inspired by their work, the ideas in \cite{calegari2013irreducibility} and \cite{ramakrishnan2013decomposition}, and the classification of automorphic representations of $\GL_4$ in \cite{asgari2007cuspidality}, we give an unconditional proof of the non-essentially-self-dual case for $n=4$. We can then deduce parts of local-global compatibility at $\ell=p$ as a by-product (see Corollary \ref{p-adicHT}). This proof should also work over any number field, if one knows the existence of the associated Galois representations over all number fields, without necessarily knowing the local-global compatibility at $\ell=p$. 
Let us state our first main theorem (Corollary \ref{mainIrr}):
\smallskip\\
\textbf{Theorem A.}\label{TheoA} Let $K$ be totally real and $\pi$ be a regular algebraic cuspidal automorphic representation of $\GL_4(\AA_K)$ and assume that $\pi$ is not essentially self-dual. Then $\rho_{\pi,\fp}$ is irreducible for any finite place $\fp$ of $E$.
\medskip \\
Our second objective in this paper is to compute the $\overline{\QQ_p}$-monodromy group $\mathcal{G}$ and the  monodromy group $G$ of 4-dimensional automorphic Galois representations in both the (essentially) self-dual and non-self-dual cases. For the group $\GL_2$, this has been done by Ribet \cite{ribet1980twists}, Momose \cite{momose1981adic}, and Nekov{\'a}{\v{r}} \cite{nekovavr2012level}. For $\GL_3$, this has been worked out by the author \cite{shavali2025image}. Here, we will study the $\GL_4$ case. 

In fact, computing $\mathcal{G}/\overline{\QQ_p}$ (or rather its Lie algebra) based on the functorial behavior of $\pi$ is straightforward and can essentially be extracted from the literature as we will see in Section \ref{LieClassification}. To compute $G$, we will need to generalize the theory of extra-twists to a more general setting and apply it in different primitive cases (orthogonal, symplectic, non-self-dual) of the classification of Section \ref{LieClassification}. These extra-twists will provide the data that we need to descend from $\overline{\QQ_p}$ to $\QQ_p$. In the symplectic case, we can only prove an unconditional result when $K=\QQ$ and over all but finitely many primes. In the other two cases, we have no restrictions. 

Let us state our second main theorem (Corollary \ref{mainBigIm}) over $\QQ$ for simplicity. Let $\pi$ be a regular algebraic cuspidal automorphic representation of $\GL_4(\AA_\QQ)$ that is not self-twisted. Assume that $\pi$ is neither a symmetric cube nor an Asai transfer. Let $\rho_{\pi,p}=\prod_{\fp|p}\rho_{\pi,\fp}$ be the $p$-adic Galois representation attached to $\pi$ and $G/\QQ_p$ be the Zariski closure of the image of $\rho_{\pi,p}$.
\medskip \\ 
\textbf{Theorem B.} Let $F$ be the field fixed by the extra-twists of $\pi$. Then for all but finitely many $p$ we have $$G^{\mathrm{der}}=\mathrm{Res}^{F_p}_{\QQ_p}(H_p),$$ where $H_p/F_p$ is a form of one of the groups $\mathrm{SO_4}$, $\mathrm{Sp}_4$, or $\SL_4$, based on whether $\pi$ is orthogonal, symplectic, or non-self-dual. 
\medskip \\
Note that the image would be $p$-adically open in the group $H_p(F_p)$. Therefore, this theorem is in the spirit of the results of Ribet \cite{ribet1980twists} and Momose \cite{momose1981adic}. The group $H_p$ should be thought of as the higher dimensional local analogue of the quaternion algebra of Ribet and Momose in the case of modular forms.
If one believes that there is a motive associated with $\pi$, then the group $\mathrm{Res}^{F_p}_{\QQ_p} H_p$ should be the semisimple part of the Mumford-Tate group of that motive. Since the field $F$ is a global object, it is natural to expect the existence of a global $H/F$ whose base change gives all the $H_p$'s. In fact, combining Theorem B with a famous result of Larsen \cite{larsen1995maximality} one can show that $H_p$ is often quasi-split. We also use this result to prove residual big image results for a density one set of primes (see Theorem \ref{residualBigImage}). 
\medskip \\
\noindent\textbf{Methods.} Let us briefly explain the methods used in proving Theorems A and B. One shows Theorem A by contradiction. Let $\rho_{\pi,\fp} = \sigma \oplus \tau$ be reducible. If at least one of $\sigma$ or $\tau$ is one dimensional, then it has to be locally algebraic by the results of \cite{bockle2025weak}. Then one can use a standard argument using exterior square $L$-functions to get a contradiction. This is essentially the same proof as in \cite{calegari2013irreducibility}. The difficult case is when both $\sigma$ or $\tau$ are two dimensional irreducible representations. In this case, we first transfer our representations to $\GL_6$ using Langlands functoriality for $\wedge^2:\GL_4 \rightarrow \GL_6$ established by Kim \cite{kim2003functoriality}. A result of Asgari and Raghuram then helps us to determine when this transfer is itself cuspidal \cite{asgari2007cuspidality}. Establishing the cuspidality of this transfer, we then use another transfer, this time from $\GL_6$ to $\GL_{20}$ using $\wedge^3:\GL_6 \rightarrow \GL_{20}$ (the idea of using this map was already considered by Ramakrishnan \cite{ramakrishnan2013decomposition}). The problem is that Langlands functoriality for this map has not been established yet. Nevertheless, one knows the expected behavior of the exterior cube $L$-function by \cite{ginzburg2000exterior}. In particular, a question raised by Ginzburg and Rallis about determining when this $L$-function has a pole at $s=1$ is answered by Yamana in \cite{yamana2015poles} which will be crucial to our argument. Combining all of these results and some formal computations with $L$-functions, we are able to reduce to the case where $\wedge^2 \pi$ is of the form $\mathrm{Ind^{K'}_K}(\pi')$ for a totally real quadratic field extension $K'/K$ and cuspidal automorphic representation $\pi'$ of $\GL_3(\AA_{K'})$. It is then easy to check that $\pi'$ must be regular algebraic, use the existence of Galois representations attached to such automorphic representations, and use the irreducibility results of Böckle and Hui for $\GL_3$ \cite{bockle2025weak} to deduce the result. \\
When one established irreducibility, one can essentially determine the Lie algebra of the image over $\overline{\QQ_p}$ from the functoriality properties of $\pi$, since there are not so many cases for $\GL_4$. See the table at the end of Section \ref{LieClassification}. There are three cases that do not reduced to $\GL_1$ and $\GL_2$, corresponding to simple Lie algebras $\mathfrak{sp}_4$, $\mathfrak{so}_4$, and $\mathfrak{sl}_4$. In these cases, the image of the Galois representation is contained in $\mathrm{GSp}_4(\overline{\QQ_p})$, $\mathrm{GO}_4(\overline{\QQ_p})$, and $\mathrm{GL}_4(\overline{\QQ_p})$, respectively. All of these three reductive groups have the property that the center of the group is one dimensional. Under this condition, we are able to generalize Ribet's theory of inner-twists for modular forms, to the extra-twists for such reductive groups. We develop the theory on the Galois side since we want to use it to compute the image. These extra-twists should be thought of as symmetries that are closely related to the endomorphisms of the hypothetical motive that should be associated with the representation $\pi$. With the right theory of extra-twists, one can give explicit formulas for the field fixed by all of these symmetries (see Lemma \ref{cent}) generalizing those of Ribet. This is the essential technical heart of proving open image results. The rest of the proof is similar to the one given in \cite{shavali2025image} for the case of $\GL_3$, using the Lie algebra version of Goursat's lemma following Ribet.  
\medskip \\
\noindent\textbf{The structure of this paper.} In Sections \ref{SectionPreliminaries} and \ref{SectionLocally}, we review some background material and prove some lemmas that we are going to use often in the paper. In Section \ref{SectionIrreducibility}, we prove our main irreducibility result. Most of the arguments here are automorphic in nature. As a result of our irreducibility theorem, we can in fact deduce local-global compatibility at $\ell=p$ for a density one set of primes. Section \ref{LieClassification} is devoted to classifying automorphic representations of $\GL_4$ based on their functoriality behavior and computing the Lie algebra (over $\overline{\QQ_p}$) of the attached Galois representations. In Section \ref{SectionExtra} we develop the theory of extra-twists for any split reductive group that admits an irreducible faithful representation. We use them in Section \ref{SectionBig} to prove our main big image results. 
\medskip \\
\noindent\textbf{Acknowledgments.} This project benefited enormously from many conversations with Gebhard Böckle and Chun-Yin Hui, to whom I am deeply grateful. I would also like to thank Frank Calegari and Toby Gee for their helpful comments, Jack Thorne for a discussion about local–global compatibility, and Judith Ludwig and Dimitri Whitmore for their interest in this work. This work was supported by the Deutsche Forschungsgemeinschaft (DFG) through the Collaborative Research Centre TRR 326 Geometry and Arithmetic of Uniformized Structures, project number 444845124. 
\newpage
\noindent \textbf{Notations.} We will use the following notations and conventions throughout the paper. 
\begin{itemize}
    \item[--] For a number field $K$, $\Gamma_K$ denotes the absolute Galois group and $\AA_K$ denotes the ring of adeles over $K$. Also, we use $\Sigma_K$ for the set of finite places of $K$. 
    \item[--] We use the notation $\pi_1 \boxplus \pi_2$ for the isobaric sum and $\pi_1 \boxtimes \pi_2$ for the automorphic tensor product of $\pi_1$ and $\pi_2$. If $\pi_2=\chi$ is a Hecke character, we simply write $\pi_1 \otimes \chi$ for the automorphic tensor. 
    \item[--] Let $H$ be a subgroup of the abstract group $G$ and $\rho$ and $\mu$ be representations of $G$ and $H$, respectively. We write $\mathrm{BC}^G_H(\rho)$ or simply $\rho|_{H}$ for the restriction of $\rho$ to $H$, and $\mathrm{Ind}^G_H (\mu)$ for the induction of $\mu$ to $G$. If $G=\Gamma_K$ and $H=\Gamma_L$ for an extension $L/K$, then we also write $\mathrm{BC}^L_K$ and $\mathrm{Ind}^L_K$ for $\mathrm{BC}^{\Gamma_K}_{\Gamma_L}$ and $\mathrm{Ind}^{\Gamma_K}_{\Gamma_L}$. We use the same notations for the automorphic base change and automorphic induction.
    \item[--] For an automorphic representation $\pi$ of a split reductive group $G$ and a representation $r:G \rightarrow \GL_n$, we write $L(s,\pi,r)$ for the associated $L$-function. If $G=\GL_n$ and $r$ is the standard representation, we simply write $L(s,\pi)$. 
    We always normalize our $L$-functions to have critical strip $0<\mathrm{Re}(s)<1$.   
    \item[--] For an algebraic group $G$ over a field $F$ and a field extension $E/F$, $\mathrm{Aut}_E(G)$ denotes the groups of automorphisms of $G_E$.  $\mathrm{Inn}_E(G)$ denotes the subgroup of the inner automorphisms of $G$ and $\mathrm{Out}_E(G)\coloneqq\mathrm{Aut}_E(G)/\mathrm{Inn}_E(G)$. When $E=F$ we drop it from the notation.
    \item[--] By the density of a subset $\mathcal{P}$ of the set of rational prime numbers $\mathrm{Primes}$, we always mean the Dirichlet density, that is 
    \[
    \lim_{s \rightarrow 1^+} \frac{\sum_{p \in \mathcal{P}} p^{-s}}{\sum_{p \in \mathrm{Primes}} p^{-s}}.
    \]
    \item[--] We have different ways of writing down a compatible family of Galois representations. Primarily, an $E$-rational compatible family $\{ \rho_\fp\}_{\fp \in \Sigma_E}$ is given by one representation 
    \[
    \rho_\fp:\Gamma_K \rightarrow \GL_n(\overline{E_\fp})
    \]
    for each finite place $\fp$ of $E$, such that the characteristic polynomials at almost all Frobenius elements of $\Gamma_K$ are compatible in the usual way. We often like to gather all finite places $\fp$ above a rational prime $p$ together and form the family $\{\rho_p\}_{p \in \Sigma_\QQ}$ such that
    \[
    \prod_{\fp|p} \rho_\fp= \rho_p:\Gamma_K \rightarrow \GL_n\left(\prod_{\fp|p} \overline{E_\fp}\right).
    \]
    Since each embedding $\lambda : E \hookrightarrow \overline{\QQ_p}$ induces an absolute value on $E$ (which corresponds to some finite place above $p$), we can also form a family $\{\rho_\lambda\}_\lambda$ of the form
    \[
    \rho_\lambda: \Gamma_K \rightarrow \GL_n(\overline{\QQ_p})
    \]
    for all primes $p$ and embeddings $\lambda: E \hookrightarrow \overline{\QQ_p}$. 
\end{itemize}
\section{Preliminaries}\label{SectionPreliminaries}

\subsection{Some Automorphic Background}

In this section, we review some of the basic definitions and results that we are going to use throughout the paper. Let $K$ and $E$ be two number fields. 

We usually work with regular algebraic cuspidal (or RAC for short) automorphic representations of $\GL_n(\AA_K)$. A cuspidal automorphic representation of $\GL_n(\AA_K)$ is regular algebraic if its archimedean component $\pi_\infty$ is infinitesimally equivalent to an irreducible algebraic representation of $\mathrm{Res}^K_\QQ \GL_n$. Therefore, there exists a unique set of weights 
$$a=(a_{\tau,i})_{\tau,i} \in (\ZZ^n)^{\mathrm{Hom}(K,\CC)}$$
satisfying 
\[
a_{\tau,1} \geq \cdots \geq a_{\tau,n},
\]
such that at each embedding $\tau$, this algebraic representation is equivalent to the dual of the irreducible representation of $\GL_n$ with highest weight $a_\tau = (a_{\tau,i})_i$. 

Let us recall different cases of Langlands functoriality that we are going to use later. The first result is the functoriality of 
\[
\boxtimes : \GL_2 \times \GL_2 \rightarrow \GL_4,
\]
which is essentially due to Ramakrishnan \cite{ramakrishnan2000modularity}.

\begin{theo}
    Let $\pi_1$ and $\pi_2$ be isobaric automorphic representations of $\GL_2(\AA_K)$. Then the admissible representation $\pi_1 \boxtimes\pi_2$, defined locally using the local Langlands correspondence, is an isobaric automorphic representation. 
\end{theo}

We will also use the functoriality of 
$$\wedge^2: \GL_4 \rightarrow \GL_6$$
due to Kim \cite{kim2003functoriality}. Let $\pi = \otimes'\pi_v$ be an automorphic representation of $\GL_4(\AA_K)$. Using the local Langlands correspondence, one can define an irreducible admissible representation $\wedge^2 \pi_v$ for each $v$. Then we have:
\begin{theo}\label{Kim}
    Let $\pi = \otimes'\pi_v$ be a cuspidal automorphic representation $\GL_4(\AA_K)$. Then there exists an automorphic representation $\wedge^2 \pi$ of $\GL_6(\AA_K)$ such that for any place $v$ of $K$ not above 2 and 3 one has $(\wedge^2 \pi)_v = \wedge^2(\pi_v)$.  
\end{theo}
The lack of local-global compatibility at $2$ and $3$ is harmless for our global arguments. So we will ignore this issue later on. It is natural to ask if one can predict when $\wedge^2 \pi$ would be cuspidal. This has been worked out by Asgari and Raghuram \cite[Theorem 1.1]{asgari2007cuspidality}. We will use their results extensively, so let us recall their main theorem here. 

\begin{theo}\label{Asgari}
    Let $\pi$ be a cuspidal automorphic representation of $\GL_4(\AA_K)$. Then the following are equivalent:
    \begin{enumerate}
        \item[(i)] $\wedge^2 \pi$ is not cuspidal.
        \item[(ii)] $\pi$ is one of the following: 
        \begin{enumerate}
            \item[(a)] $\pi = \pi_1 \boxtimes \pi_2$ for cuspidal automorphic representations $\pi_1$ and $\pi_2$ of $\GL_2(\AA_K)$. (This may be viewed as a transfer from the split $\mathrm{GSpin_4}$ to $\GL_4$.)
            \item[(b)] $\pi = \mathrm{As}(f)$, the Asai transfer of a dihedral cuspidal automorphic representation of $\GL_2(\AA_{K'})$ for a quadratic extension $K'/K$. (This may be viewed as a transfer from the quasi-split non-split $\mathrm{GSpin}^*_4$ to $\GL_4$.)
            \item[(c)] $\pi$ is the functorial transfer of a globally generic cuspidal automorphic representation of $\mathrm{GSp}_4$ to $\GL_4$.
            \item[(d)] $\pi = \mathrm{Ind}^{K'}_K (f)$ for a cuspidal automorphic representation $f$ of $\GL_2(\AA_{K'})$ of a quadratic extension $K'/K$.
        \end{enumerate}
        \item[(iii)] $\pi$ satisfies one of the following:
        \begin{enumerate}
            \item[($\alpha$)] $\pi = \pi^\vee \otimes \chi$ for some Hecke character $\chi$ of $K$, and $\pi$ is not the Asai transfer of a non-dihedral cuspidal automorphic representation.  
            \item[($\beta$)] $\pi = \pi \otimes \chi$ for some Hecke character $\chi$ of $K$.
        \end{enumerate}
    \end{enumerate}
\end{theo}

\subsection{Preliminary Lemmas}
In this section, we will prove a few preliminary lemmas that are going to be used throughout the paper. We start with the types of reductive groups that we will use as the target of our Galois representations.
\begin{lemma}\label{center}
    Let $\widehat{G}$ be a reductive group over an algebraically closed field of characteristic 0. Then $\hG$ admits an irreducible faithful representation if and only if the center of $\hG$ embeds into $\GG_m$.
\end{lemma}
\begin{proof}
    One direction is easy. Assume that 
    \[
    \hG \hookrightarrow \GL_n
    \]
    is a faithful irreducible representation. Then Schur's lemma implies that the center $\hZ$ of $\hG$ consists only of scalar matrices, which gives an embedding $\hZ \hookrightarrow \GG_m$. For the other direction see \cite{McN}.
\end{proof}

Now assume that $p$ is a prime number, $K$ and $E$ are number fields, and $\hG/\QQ$ is a split reductive group admitting an irreducible faithful representation. Let 
\[
E_p \coloneqq E \otimes_\QQ \QQ_p = \prod_{\fp|p} E_\fp. 
\]
We will consider Galois representations of the form
\[
\prod _{\fp|p}\rho_\fp=\rho: \Gamma_K \rightarrow \hG(E_p) = \prod _{\fp|p} \hG(E_\fp),
\]
where $\Gamma_K = \mathrm{Gal}(\overline{K}/K)$. We can enlarge $E$ to make sure that the irreducible faithful representation of $\hG$ is defined over $E$. Therefore, we can also assume that we have an (irreducible) embedding
\[
\widehat{G}(E_p) \hookrightarrow \GL_n(E_p)
\]
from now on. Two Galois representations $\rho_1$ and $\rho_2$ as above are said to be isomorphic or conjugate if there exists a matrix $A \in \GL_n(E_p)$ such that 
$$\rho_1 = A\cdot \rho_2 \cdot A^{-1}.$$ 
Finally, we say that a Galois representation $\rho$ is absolutely irreducible, if for each prime ideal $\fp$ of $E$ above $p$, the representation $\rho_\fp$ is absolutely irreducible. 
\begin{lemma}\label{charTwist}
    Let $L$ be a finite Galois extension of $K$ and $\rho_1,\rho_2:\Gamma_K \rightarrow \widehat{G}(E_p)$ absolutely irreducible $p$-adic Galois representations such that $\rho_1|_{\Gamma_L}\simeq \rho_2|_{\Gamma_L}$. Then there exists a ($\widehat{Z}$-valued) Galois character $\chi$ such that
    \[
    \rho_1 \simeq \rho_2 \otimes \chi.
    \]
\end{lemma}
\begin{proof}
    Conjugating $\rho_1$, we may assume that $\rho_1$ and $\rho_2$
    are in fact equal when restricted to $\Gamma_L$. Now define
    $$
    \phi (g) := \rho_{1}^{-1}(g)\rho_{2}(g).
    $$
    A priori $\phi$ is just a map
    $\phi:\Gamma_K \rightarrow \widehat{G}(E_p)$ which
    is trivial on $\Gamma_L$. We want to prove that it
    is actually a homomorphism with values
    in the center (hence a character) on all of $\Gamma_K$. 

    Let $g\in \Gamma_K$ and $h \in \Gamma_L$. 
    Note that
    $\rho_{1} (h) = \rho_{2} (h)$
    and 
     $\rho_{1} (ghg^{-1}) = \rho_{2} (ghg^{-1})$
     since $\Gamma_L$ is normal in $\Gamma_K$.
     Now the following computation shows that 
     $\phi (g)=\rho_{1}^{-1}(g)\rho_{2}(g)$
     commutes with $\rho_{2} (h)$:
     $$
     \rho_{1}^{-1}(g)\rho_{2}(g)\rho_{2} (h)
     = \rho_{1}^{-1}(g)\rho_{2}(gh) =
     \rho_{1}^{-1}(g)
     \rho_{2}(ghg^{-1})\rho_{2}(g)
     $$
     $$
     =\rho_{1}(g^{-1})
     \rho_{1}(ghg^{-1})\rho_{2}(g)
     =\rho_{1} (h)\rho_{1}^{-1}(g)\rho_{2}(g)
     = \rho_{2} (h)\rho_{1}^{-1}(g)\rho_{2}(g)
     $$
     Now, since $\rho_2$ is absolutely irreducible when restricted to $\Gamma_L$, Schur's lemma implies that $\phi(g)$ is scalar and it commutes with everything. By its definition, $\phi(g)$ has values in $\widehat{G}$, so it must land in the center of $\widehat{G}$. 
\end{proof}

One can also prove an automorphic analogue of this lemma. We will only use this for the group $\GL_4$ and $[L:K]=2$. 

\begin{lemma}
    Let $\pi_1$ and $\pi_2$ be cuspidal automorphic representations of $\GL_n(\AA_K)$ and let $L/K$ be a cyclic prime degree extension. Assume that 
    \[
    \mathrm{BC}_{L/K}(\pi_1) = \mathrm{BC}_{L/K}(\pi_2).
    \]
    Then, there exists a Hecke character $\chi$ such that $\pi_1 = \pi_2 \otimes \chi$.
\end{lemma}
\begin{proof}
    Note that for cyclic prime degree extensions, automorphic base change and automorphic induction are known by \cite{arthur1989simple}. Since $L/K$ is cyclic, all irreducible representations of $\mathrm{Gal}(L/K)$ are characters. Let $\chi_1,\cdots,\chi_m$ be all the Hecke characters of $K$ corresponding to these characters, i.e. they become trivial after base change to $L$. In other words, these characters factor through the image of the norm map:
    \[
    \chi_i : K^\times\backslash \AA_K^\times /\mathrm{Nr}_{L/K}(\AA_L^\times) \rightarrow \CC^\times.
    \]
    Now applying induction we have
    \[
    \mathrm{Ind}_{L/K} (\mathrm{BC}_{L/K}(\pi_1)) = \mathrm{Ind}_{L/K} (\mathrm{BC}_{L/K}(\pi_2)).
    \]
    Frobenius reciprocity implies 
    \[
    \pi_1 \boxtimes \mathrm{Ind}_{L/K}(\mathbf{1}) = \pi_2 \boxtimes \mathrm{Ind}_{L/K}(\mathbf{1}).
    \]
    Therefore we obtain the equality of the isobaric sums
    \[
    (\pi_1 \otimes \chi_1) \boxplus \cdots \boxplus (\pi_1 \otimes \chi_m) = (\pi_2 \otimes \chi_1) \boxplus \cdots \boxplus (\pi_2 \otimes \chi_m). 
    \]
    Since $\pi_1$ and $\pi_2$ are cuspidal, this is an isobaric decomposition and the summands of both sides should coincide up to permutation. This implies the result.
\end{proof}

We will use \cite[Lemma 4.3]{calegari2013irreducibility} in several different places. In fact, at one point we need a slight improvement of this lemma. Let us state and prove this slight generalization. The proof is identical to the one in \cite{calegari2013irreducibility}.
\begin{lemma}\label{GerenalizedCG}
    Suppose that the group $G$ acts irreducibly on a finite dimensional vector space $V$. Let $G'$ be a finite index normal subgroup of $G$ and let $V = \oplus_{i=1}^k W_i$ be the decomposition of $V$ into isotypic representations of $G'$. Then either $k=1$ or there exists a proper subgroup $H$ of $G$ containing $G'$ and an isotypic representation $W$ of $H$ such that $V \simeq \mathrm{Ind}^G_H(W)$. Moreover, we can assume $\dim W = \dim W_1$.
\end{lemma}
\begin{proof}
    Since $G'$ is normal and $W_i$'s are  isotypic representations of $G'$ of different types, the group $G$ acts on the set $\{W_1,\dots,W_k\}$. This action is transitive, because if it had a proper block, this block would give a non-trivial proper subspace of $V$ that is stable under $G$, which contradicts irreducibility of $V$. In particular, all $W_i$'s have the same dimension. Let $H$ be the stabilizer of $W_1$, hence $W_1$ is also a representation of $H$ and $\dim (W_1) = \dim (V)/k=\dim (V)/[G:H]$. Therefore, if $H=G$, then $k=1$ and we are done. So we can assume that $H$ is a proper subgroup. \\
    Now note that we have a non-trivial $H$-equivariant projection $V \rightarrow W_1$. By Frobenius reciprocity, we obtain that there is a non-trivial $G$-equivariant map 
    \[
    V \rightarrow \mathrm{Ind}^G_H(W_1).
    \]
    Since $V$ is irreducible as a representation of $G$, this map is injective. Since both sides have the same dimension, this map must be an isomorphism. 
\end{proof}

For any ring $A$, a Galois representation $\rho: \Gamma_K \rightarrow \hG(A)$ is called potentially abelian if there exists a finite extension $L/K$ such that $\rho(\Gamma_L)$ is an abelian subgroup of $\hG(A)$.

\begin{cor}
    Let $\rho : \Gamma_K \rightarrow \GL_n(\overline{\QQ_p})$ be a Hodge-Tate irreducible Galois representation with distinct Hodge-Tate weights. If $\rho$ is potentially abelian, then there exists a degree $n$ extension $L/K$ and a character $\chi$ of $\Gamma_L$ such that 
    \[
    \rho = \mathrm{Ind}^L_{K}(\chi).
    \]
\end{cor}
\begin{proof}
    There exists a finite extension $L'/K$ such that $\rho|_{\Gamma_{L'}}$ has abelian image. Therefore, it decomposes as a sum of characters
    \[
    \rho|_{\Gamma_{L'}} = \chi_1 \oplus \chi_2 \oplus \cdots \oplus \chi_n. 
    \]
    Since $\rho$ is regular (has distinct Hodge-Tate weights), these characters are distinct. Therefore, we are in the situation of Lemma \ref{GerenalizedCG} with $W_1=\chi_1$. This implies the result. 
\end{proof}

Finally, we need some information about two dimensional automorphic Galois representations, over both totally real and CM fields. Recall that a Galois representation is called strongly irreducible if its base change to any finite extension is irreducible. This is equivalent to the Lie algebra of the image of the representation being irreducible. 

\begin{lemma}\label{GL(2)}
    Let $K$ be a totally real or CM field and $\pi$ be a regular algebraic cuspidal automorphic representation of $\GL_2(\AA_K)$. Assume that $\pi$ is non-dihedral. Then for every finite place $\fp$ of $\QQ(\pi)$, the Galois representation $\rho_{\pi,\fp}$ attached to $\pi$ is strongly irreducible. In particular, the semisimple part of the Lie algebra of the image of
    \[
    \rho_{\pi,\fp}:\Gamma_K \rightarrow \GL_2(\overline{E_\fp})
    \]
    is equal to $\mathfrak{sl}_2(\overline{E_\fp})$. 
\end{lemma}
\begin{proof}
The representations $\rho_{\pi,\fp}$ are irreducible and form a very weakly compatible system  by \cite[Theorem 7.1.10]{allen2023potential}. Therefore, over a set of rational primes of Dirichlet density one, $\rho_{\pi,\fp}$ is crystalline of the expected Hodge-Tate weights. In particular, it is regular. This implies that over those primes, $\rho_{\pi,\fp}$ is not a twist of an Artin representation. Now, \cite[Theorem 7.1.2]{allen2023potential} implies that for all $\fp$, the representation $\rho_{\pi,\fp}$ is strongly irreducible. In particular the tautological representation of the semisimple part of the Lie algebra of the image is irreducible. The only faithful irreducible semisimple Lie subalgebra of $\mathfrak{gl}_2$ is $\mathfrak{sl}_2$, and we are done. 
\end{proof}

\section{Locally Algebraic Representations}\label{SectionLocally}
Serre first defined the notion of locally algebraic representations, which can be seen as a precursor to more sophisticated later notions of $p$-adic Hodge theory, to study abelian Galois representations \cite{serre1972proprietes}. Böckle and Hui generalized some of the results of Serre about abelian compatible families of Galois representations to \textit{weak abelian direct summands} of such families \cite{bockle2025weak}. They used this to prove irreducibility of Galois representations of $\GL(3)$ in the non-self-dual case. In this section, we quickly review their results. 

Assume that $K$ is a number field. A $p$-adic representation of $\Gamma_K$ is called abelian if its image is an abelian group. In other words, it factors through $\Gamma_K^{ab}$, the abelianization of $\Gamma_K$:
\[
\phi_p:\Gamma_K^{ab}\rightarrow \GL_n(\overline{\QQ_p}).
\]
Such abelian representations can be related to adelic groups using the Artin reciprocity map:
\[
\mathrm{Art_K}:\AA_K^\times /K^\times \rightarrow \Gamma_K^{ab}.
\]

\begin{defi}
    An abelian semisimple $p$-adic Galois representation 
    \[
    \phi_p:\Gamma_K^{ab} \rightarrow \GL_n(\overline{\QQ_p})
    \]
    is called locally algebraic if there exists a morphism of algebraic groups
    \[
    r:T_{\overline{\QQ_p}} \rightarrow \GL_{n,\overline{\QQ_p}}
    \]
    such that the composition
    \[
    T(\QQ_p) = \prod_{v|p} K_v^\times \rightarrow \AA_K^\times /K^\times \xrightarrow{\mathrm{Art_K}} \Gamma_K^{ab} \xrightarrow{\phi_p} \GL_n(\overline{\QQ_p})
    \]
    is equal to $r|_{T(\QQ_p)}$ on a small enough neighborhood of the identity element on the $p$-adic Lie group $T(\QQ_p) = (K \otimes_\QQ \QQ_p)^\times$.  
\end{defi}
It is essentially due to Fontaine that for an abelian semisimple Galois representation of $\Gamma_K$, being locally algebraic, being de Rham at places above $p$, and being Hodge-Tate at places above $p$ are equivalent.
What makes the definition of locally algebraic representations interesting is the following result of Serre (see \cite{henniart1980representations}).
\begin{theo}\label{SerreLocAlg}
    For an abelian semisimple Galois representation of $\Gamma_K$, being locally algebraic is equivalent to being $E$-rational for some number field $E$. 
\end{theo}

The main new input of the work of Böckle and Hui in \cite{bockle2025weak} is to generalize  Theorem \ref{SerreLocAlg} to the weak abelian direct summands of (not necessarily abelian) semisimple Galois representations.

\begin{defi}
    Let
    \[
    \rho_p:\Gamma_K \rightarrow \GL_n(\overline{\QQ_p})
    \]
    be an arbitrary Galois representation and 
    \[
    \psi_p:\Gamma_K \rightarrow \GL_m(\overline{\QQ_p})
    \]
    be a semisimple abelian Galois representation. We say that $\psi_p$ is a weak abelian direct summand of $\rho_p$ if there exists a density one set of (rational) primes $\mathcal{L}$ such that for each $\ell \in \mathcal{L}$ and finite place $v$ of $K$ above $\ell$, the representations $\rho_p$ and $\psi_p$ are both unramified at $v$ and the characteristic polynomial of $\psi_p(\mathrm{Frob}_v)$ divides the characteristic polynomial of $\rho_p(\mathrm{Frob}_v)$. 
\end{defi}

The obvious example of the above situation is when $\psi_p$ is in fact a direct summand of $\rho_p$, but there are examples where this does not hold. Finally, here is the main result of \cite{bockle2025weak}:

\begin{theo}[Theorem 1.1 of \cite{bockle2025weak}] \label{BoeckelHuiMain}
    Let $E \subset \overline{\QQ_p}$ be a number field and
    \[
    \rho_p:\Gamma_K \rightarrow \GL_n(\overline{\QQ_p})
    \]
    be a semisimple, $E$-rational, $p$-adic Galois representation. Let the representation $\psi_p$ be a weak abelian direct summand of $\rho_p$. Then $\psi_p$ is locally algebraic and hence de Rham at all places  above $p$. 
\end{theo}
We will also make use of a recent preprint of Böckle and Hui \cite{bockle2026on}. Let us state the two results from their work that we will use frequently. The first result shows that compatible families can always be realized over a number field: 
\begin{theo}[Theorem 1.3 of \cite{bockle2026on}]\label{FieldSplitsAll}
    Let $K$ be a number field and $$\{\rho_p:\Gamma_K \rightarrow \GL_n(\overline{E'_p})\}_{p \in \Sigma_\QQ}$$ a semisimple $E'$-rational compatible system of Galois representations. Then there exists a finite Galois extension $E/E'$ such that this compatible family can be realized over $E$, i.e. each $\rho_p$ may be assumed to have values in $\GL_n(E_p)$. 
\end{theo}

We will use this result without explicitly mentioning it again. So we freely assume that our Galois representations have values in completions of a number field whenever we need it. 
The second result is about comparing the decomposition of a compatible family into irreducible representations and the decomposition of the residual representations of the family. We only need a special case of this result:

\begin{theo}[Theorem 1.4 of \cite{bockle2026on}]\label{residualIrr}
    Let $K$ be a number field and $$\{\rho_\fp:\Gamma_K \rightarrow \GL_n(\overline{E_\fp})\}_{\fp \in \Sigma_E}$$ a semisimple $E$-rational compatible system of Galois representations. Assume that for a density one set of rational primes $\mathcal{P'}$, $\rho_{\fp}$ is absolutely irreducible if $\fp$ is above $p \in \mathcal{P'}$. Then, there exists a density one set of rational primes $\mathcal{P}$, such that $\overline{\rho_{\fp}}$ is absolutely irreducible if $\fp$ is above $p \in \mathcal{P}$.
\end{theo}

\section{Irreducibility of Non-self-dual Representations}\label{SectionIrreducibility}

In this section we prove our main irreducibility result. Throughout this section, we assume that $K$ is a totally real number field. 

Let $\pi$ be a regular algebraic cuspidal automorphic representation of $GL_4(\AA_K)$. Assume that $\pi$ is neither a self-twist nor essentially self-dual. Recall from the introduction that there exists a compatible family of Galois representations 
$$\{ \rho_{\pi,\fp}:\Gamma_K\rightarrow \GL_4(\overline{\QQ_p})\}_{\fp \in \Sigma_{\QQ(\pi)}}$$
attached to it. 
The goal of this section is to prove that each $\rho:=\rho_{\pi,\fp}$ is irreducible. Assume that $\rho = \sigma \oplus \tau$ (the decomposition of course depends on $\fp$) and without any loss of generality suppose that $\dim \sigma \geq \dim \tau$. Then we either have $\dim \sigma =3$ (the $(3,1)$-case) or $\dim \sigma = 2$ (the $(2,2)$-case). 
\subsection{The $(3,1)$-case}
This case is much easier. Except for using Theorem \ref{BoeckelHuiMain} instead of local-global compatibility at $\ell=p$, the proof is the exact same proof as in the work of Calegari and Gee \cite[Lemma 8.4]{calegari2013irreducibility}. 
\begin{prop}\label{not31}
    Keeping the notation and the assumptions as above, we cannot have $\dim \tau = 1$.
\end{prop}
\begin{proof}
    $\tau$ is a character and hence abelian. Therefore, by Theorem \ref{BoeckelHuiMain}, it is $E$-rational and hence locally algebraic. By class field theory, $\tau$ corresponds to some algebraic Hecke character $\eta$. Taking exterior square we have
\[
\wedge^2\rho = (\wedge^2 \sigma) \oplus (\sigma \otimes \tau).
\]
Since $\sigma$ is 3-dimensional, we have a pairing
\[
\sigma \times \wedge^2 \sigma \rightarrow \wedge^3\sigma = \det (\sigma),
\]
which gives the duality $\wedge^2 \sigma \simeq \sigma^\vee \otimes \det(\sigma)$. Note that $\det(\sigma) = \wedge^3\sigma$ is a direct summand of the $E$-rational Galois representation $\wedge^3 \rho$. Therefore, it is again locally algebraic by Theorem \ref{BoeckelHuiMain}.
By class field theory, $\det(\sigma)$ corresponds to an algebraic Hecke character $\chi$. Therefore, we have
\[
\wedge^2\rho = (\sigma^\vee \otimes \det(\sigma)) \oplus (\sigma \otimes \tau).
\]
Now, in the ring of virtual representations, one has $\sigma=\rho - \tau$. Hence, in this ring we can write
\[
\wedge^2 \rho = ((\rho^\vee-\tau^\vee) \cdot\det(\sigma)) + (\rho - \tau)\cdot \tau, 
\]
which means
\[
\wedge^2 \rho + \tau^{-1}\cdot\det (\sigma) + \tau^2 = \rho^\vee \cdot\det(\sigma) + \rho \cdot \tau.
\]
Now, twisting by $\tau^{-2}$ we get
\[
\wedge^2 \rho \cdot \tau^{-2} + \tau^{-3}\cdot\det (\sigma) + 1 = \rho^\vee \cdot\det(\sigma)\cdot\tau^{-2} + \rho \cdot \tau^{-1}.
\]
Taking $L$-functions of both sides and replacing Galois $L$-functions by the corresponding automorphic $L$-functions, we conclude: 
\[
L(s,\wedge^2\pi \otimes \eta^{-2})L(s,\eta^{-3}\chi) \zeta(s) = L(s,\pi^\vee\otimes\chi \eta^{-2}) L(s,\pi\otimes \eta^{-1}).
\]
The right hand side is an entire function, so on the left hand side the pole of $\zeta$ at $s=1$ should be canceled. $L(s,\eta^{-3}\chi)$ is the Hecke $L$-function of an algebraic character and is non-zero at $s=1$. Finally, the term $L(s,\wedge^2\pi \otimes \eta^{-2},s)$ is non-zero at $s=1$ by \cite[Theorem 1.1]{shahidi1997non}. This is a contradiction. 
\end{proof}

\subsection{The $(2,2)$-case}
This case is more complicated. We break the proof into a few steps. The first proposition contains all the $L$-function arguments that we make in this case:

\begin{prop}\label{isInduction}
    Keeping the notation and the assumptions as in the beginning of this section, suppose that $\dim \sigma = 2$. Then $\det (\sigma)^6=\det(\tau)^6$, but $\det(\sigma)^3 \neq \det (\tau)^3$. Moreover, If the degree two extension $K'/K$ is the splitting field of the character $\det(\sigma)^3 \cdot \det (\tau)^{-3}$, then there exists a cuspidal automorphic representation $\pi'$ of $\GL_3(\AA_{K'})$ such that
    \[
    \wedge^2 \pi = \mathrm{Ind}^{K'}_K(\pi').
    \]
\end{prop}
\begin{proof}
We have $\rho=\sigma \oplus \tau$ for two-dimensional summands $\sigma$ and $\tau$. Taking exterior squares of both sides, we obtain
\[
\wedge^2 \rho = \det (\sigma) \oplus \det (\tau) \oplus (\sigma \otimes \tau).
\]
By Theorem \ref{BoeckelHuiMain}, this tells us that both $\det(\sigma)$ and $\det(\tau)$ are locally algebraic. Let the algebraic Hecke character $\chi$ correspond to $\det (\sigma)$ and $\eta$ to $\det (\tau)$, by class field theory. 
Now, recall that by Theorem \ref{Kim},
$\wedge^2\pi$ exists as an automorphic representation and $\wedge^2 \rho$ is the Galois representation associated with it (except that we might not know local-global compatibility at $p=2,3$ which is not important for the global arguments, so we ignore it). \\
On the other hand, since we assumed that $\pi$ is not essentially self-dual and not a self-twist, by Theorem \ref{Asgari}, we know that $\wedge^2 \pi$ is  cuspidal. Let us set $\Pi:=\wedge^2 \pi$ and $\beta=\sigma \otimes \tau$ to ease the notation. So we can write:
\begin{equation}\label{eq114}
    \rho_\Pi =  \det (\sigma) \oplus \det (\tau) \oplus \beta.
\end{equation}
Now if we apply $\wedge^2$ again, we get
\[
\wedge^2 \rho_\Pi = \wedge^2\beta \oplus (\beta \otimes \det(\sigma)) \oplus (\beta \otimes \det (\tau)) \oplus (\det (\sigma) \cdot \det (\tau)).
\]
Let us replace $\det(\sigma)$ and $\det (\tau)$ by their corresponding Hecke characters $\chi$ and $\eta$ by abuse of notation, for simplicity. Therefore:
\[
\wedge^2 \rho_\Pi = \wedge^2\beta \oplus (\beta \otimes \chi) \oplus (\beta \otimes \eta) \oplus (\chi \cdot \eta).
\]
So, in the ring of virtual representations, we find the following expression for $\wedge^2 \beta$:
\[
\wedge^2 \beta = \wedge^2 \rho_\Pi - (\beta \cdot \chi) - (\beta \cdot \eta) - \chi \cdot \eta.
\]
On the other hand, in this ring we also have $\beta = \rho_\Pi - \chi - \eta$. Hence, replacing all the $\beta$'s on the right hand side, we get
\begin{equation}
\begin{aligned}\label{wedge2}
\wedge^2 \beta = \wedge^2 \rho_\Pi - \rho_\Pi \cdot \chi + \chi^2 + \eta \cdot \chi - \rho_\Pi \cdot \eta + \chi \cdot \eta + \eta^2 - \chi \cdot \eta \\ = \wedge^2 \rho_\Pi - \rho_\Pi \cdot \chi - \rho_\Pi \cdot \eta + \chi^2 + \eta^2 + \chi \cdot \eta .
\end{aligned}
\end{equation}
This will be useful later because it virtually gives an expression for $\wedge^2 \beta$ in terms of  Galois representations that we know are automorphic. 
\\
Now we will look at the $\wedge^3$ of the decomposition (\ref{eq114}) of our 6-dimensional representation:
\[
\wedge^3 \rho_\Pi = \wedge^3 \beta \oplus (\wedge^2 \beta \otimes \chi) \oplus (\wedge^2 \beta \otimes \eta) \oplus (\beta \otimes \chi \otimes \eta)
\]
(all other terms that have $\wedge^{\geq 2}$ of a character vanish). Note that both sides are of dimension 20. Also, we have a pairing $$\wedge^3 \beta \times \beta \rightarrow \wedge^4 \beta = \det (\beta),$$ since $\beta$ is 4 dimensional. This gives us the duality $\wedge^3 \beta = \beta^\vee \otimes \det(\beta)= \beta^\vee \otimes \chi^2 \eta^2$. We then replace all $\wedge^2 \beta$'s by equation (\ref{wedge2}) and all $\beta$'s by $\rho_\Pi-\chi-\eta$ in the ring of virtual representations. We get:
\begin{align*}
\wedge^3 \rho_\Pi = (\rho_\Pi-\chi-\eta)^\vee \cdot \chi^2 \eta^2 \\ + \left( \wedge^2 \rho_\Pi - \rho_\Pi \cdot \chi - \rho_\Pi \cdot \eta + \chi^2 + \eta^2 + \chi \cdot \eta \right)\cdot \chi \\ +  \left( \wedge^2 \rho_\Pi - \rho_\Pi \cdot \chi - \rho_\Pi \cdot \eta + \chi^2 + \eta^2 + \chi \cdot \eta \right)\cdot \eta \\ +  \left( \rho_\Pi-\chi-\eta \right)\cdot \chi\eta
\end{align*}
Replacing $(\rho_\Pi-\chi-\eta)^\vee = \rho_{\Pi^\vee}-\chi^{-1}-\eta^{-1}$ and multiplying we get
\begin{align*}
\wedge^3 \rho_\Pi = \rho_{\Pi^\vee}\cdot \chi^2 \eta^2- \chi \eta^2- \chi^2 \eta  \\ +  \wedge^2 \rho_\Pi \cdot \chi - \rho_\Pi \cdot \chi^2 - \rho_\Pi \cdot \eta\chi + \chi^3 + \eta^2 \chi + \chi^2  \eta  \\ +   \wedge^2 \rho_\Pi \cdot \eta - \rho_\Pi \cdot \chi\eta - \rho_\Pi \cdot \eta^2 + \chi^2 \eta + \eta^3 + \chi  \eta^2  \\ +   \rho_\Pi \cdot \chi \eta-\chi^2 \eta- \chi\eta^2 
\end{align*}
After cancellation and bringing all the negative terms to the left hand side, we obtain:
\begin{align*}
\wedge^3 \rho_\Pi + \rho_\Pi \cdot \chi^2 + 2 \rho_\Pi \cdot \eta\chi + \rho_\Pi \cdot \eta^2 \\ = \rho_{\Pi^\vee}\cdot \chi^2 \eta^2 + \wedge^2 \rho_\Pi \cdot \chi + \wedge^2 \rho_\Pi \cdot \eta + \chi^3 + \eta^3 + \rho_\Pi \cdot \chi\eta 
\end{align*}
Now we twist with $\chi^{-3}$ and consider the $L$-functions of both sides:
\begin{align*}
L(s,\Pi, \wedge^3 \otimes \chi^{-3})L(s,\Pi \otimes \chi^{-1})L(s,\Pi\otimes \eta \chi^{-2}) L(s,\Pi\otimes \eta^2 \chi^{-3}) \\ = L(s,\Pi^\vee \otimes \chi^{-1}\eta^2)L(s,\Pi, \wedge^2 \otimes \chi^{-2})L(s,\Pi, \wedge^2 \otimes \eta\chi^{-3}) \zeta(s) L(s,\eta^{3}\chi^{-3})
\end{align*}
There are three terms in the above equation that we do not know are automorphic. Namely, $L(s,\Pi, \wedge^3 \otimes \chi^{-3})$, $L(s,\Pi, \wedge^2 \otimes \chi^{-2})$ and $L(s,\Pi, \wedge^2 \otimes \eta\chi^{-3})$. The latter two that appear on the right-hand side of the equation are non-zero at $s=1$ by \cite[Theorem 1.1]{shahidi1997non}. Therefore, the pole of the Riemann zeta function on the right-hand side is not canceled, since all the other terms on the right-hand side are $L$-functions of cuspidal automorphic representations and hence non-vanishing at $s=1$. This implies that the left-hand side has a pole at $s=1$. Except for $L(s,\Pi, \wedge^3 \otimes \chi^{-3})$, the other terms on the left-hand side are $L$-functions of cuspidal automorphic representations, so they are entire functions. We conclude that $L(s,\Pi, \wedge^3 \otimes \chi^{-3})$ has a pole at $s=1$. 
Now, \cite[Theorem 3.2]{ginzburg2000exterior} implies that this only happens if one has $\omega_\Pi ^ 2 \cdot \chi^{-12}=1$ but $\omega_\Pi \chi^{-6}\neq 1$. Since the central character of $\Pi$ is $\chi^3 \eta^3$ by \ref{eq114}, this means that $\chi^6 = \eta^6$ but $\chi^3 \neq \eta^3$. This proves the first part of the proposition.\\
For the second part, the main theorem of \cite{yamana2015poles} shows that this can only happen when $\Pi$ comes from an induction $\mathrm{Ind}^{K'}_K (\Pi')$ for $\Pi'$ a cuspidal automorphic representation of $\GL_3(\AA_{K'})$ for the quadratic extension $K'/K$ associated with the quadratic character $\chi^3\eta^{-3}$. This completes the proof.
\end{proof}
So far we have shown that if $\rho=\rho_{\pi,\fp}$ decomposes into the sum of two $2$-dimensional representations as $\sigma \oplus \tau$, then  there exists a quadratic extension $K'/K$ such that $\wedge^2 \pi = \mathrm{Ind}^{K'}_K(\pi')$ for a cuspidal automorphic representation $\pi'$ of $\GL_3(\AA_{K'})$. The extension $K'$ is associated to the quadratic character $\det(\sigma)^3\det(\tau)^{-3}$. We keep this notation for the rest of this section.
\begin{lemma}\label{isTotReal}
    The number field $K'$ is totally real.
\end{lemma}
\begin{proof}
    Let $c \in \Gamma_K$ be any choice of complex conjugation. Then
    \[
    \det \rho(c)=(-1)^{\lfloor\frac{n+1}{2}\rfloor}=+1,
    \]
    since $n=4$. In fact, if one is willing to accept Arthur's work on endoscopic classification, it follows from \cite[Theorem 1.1]{caraiani2016image} that there are exactly two $+1$ and two $-1$ eigenvalues, but we do not need this. Since $\det \rho = \det \sigma \cdot \det \tau$, this implies that $c$ is in the kernel of the character $\det(\sigma)^3\det(\tau)^{-3}$ and hence the field associated with it is totally real. 
\end{proof}
\begin{lemma}\label{isRegular}
    The cuspidal automorphic representation $\pi'$ is regular algebraic.
\end{lemma}
\begin{proof}
    Let $v|\infty$ be an archimedean place of $K$ and $\pi_v$ be the $(\mathfrak{gl}_4,K)$-module attached to $\pi$ at that place. By archimedean local Langlands, there exists a representation of $W_{\RR} = \CC^\times \cup j\CC^\times$ associated with $\pi_v$. The restriction of this representation to $\CC^\times$ decomposes into the sum of characters, so it is of the form
\[
\rho_{\pi_\infty}|_{W_\CC} = \chi_1 \oplus \chi_2 \oplus\chi_3 \oplus \chi_4.
\]
Each $\chi_i$ can be written in the form $z \mapsto |z|^{\frac{-3}{2}} (z^{p_i}\overline{z}^{q_i})$. Since $\pi$ is algebraic, $p_i,q_i$ are integers, and since it is regular, the four characters $\chi_i$ are distinct. 
\\
Now, note that $\wedge^2 \pi = \mathrm{Ind}^{K'}_K (\pi')$. Therefore
\[
\wedge^2(\mathrm{BC}^{K'}_K(\pi)) = \mathrm{BC}^{K'}_K(\wedge^2 \pi) = \pi' \boxplus {}^\gamma\pi',
\]
where $\gamma$ is the non-trivial element of $\mathrm{Gal}(K'/K)$.  Since $\pi$ is not self-twisted, $\mathrm{BC}^{K'}_K(\pi)$ is itself cuspidal. But we see that its $\wedge^2$ is not, so we can apply the classification in \cite{asgari2007cuspidality}. By Theorem 1.1 (ii) of \cite{asgari2007cuspidality}, an isobaric decomposition of $\wedge^2\mathrm{BC}^{K'}_K(\pi)$ into two  automorphic representations of $\mathrm{GL}_3$ can only happen if $\mathrm{BC}^{K'}_K(\pi) = \pi_1 \boxtimes \pi_2$ and
\[
 \pi' \boxplus {}^\gamma\pi' = \wedge^2(\mathrm{BC}^{K'}_K(\pi)) = (\mathrm{sym}^2(\pi_1) \otimes \omega_2) \boxplus (\mathrm{sym}^2(\pi_2) \otimes \omega_1),
\]
where $\omega_i$ is the central character of $\pi_i$. 
Therefore, we can assume $\pi' = \mathrm{sym}^2(\pi_1) \otimes \omega_2$ without loss of generality. 
\\
Now, $\pi'$ being an isobaric summand of the algebraic automorphic representation $\mathrm{BC}^{K'}_K(\wedge^2 \pi)$, is itself algebraic. Fix an infinite place $v'|\infty$ of $K'$ above $v$ (in fact a real embedding since $K'$ is totally real). Then there is a representation of $W_\RR$ associated to $\pi'$ at $v'$, and its restriction to $W_\CC=\CC^\times$ is given as the sum of three characters as follows:
$$\rho_{\pi'_v}|_{W_\CC}=\eta_1 \oplus \eta_2 \oplus\eta_3.$$
Similarly, the $L$-parameter of $\mathrm{BC}^{K'}_K(\wedge^2 \pi)$ at $v'$ is given by the six characters $\{ \chi_i \chi_j \}_{i<j}$. Since $\chi_i$'s are distinct, at most two of these six characters are equal (and in fact two will be the same because of purity, but we do not use that). Now, since $\pi' \boxplus {}^\gamma\pi' = \wedge^2(\mathrm{BC}^{K'}_K(\pi))$, this implies that at most two of $\eta_i$'s are equal. On the other hand we have $$\pi' = \mathrm{sym}^2(\pi_1) \otimes \omega_2.$$ 
The $L$-parameter of $(\pi_i)_{v'}$ (restricted to $W_\CC$) is given by two characters $\theta_1, \theta_2$. If $\theta_1=\theta_2$, then the $L$-parameter of $\mathrm{sym}^2(\pi_1)_{v'}$ is given by three equal characters which cannot happen since at most two of $\eta_i$'s are equal. Therefore, $\theta_1 \neq \theta_2$ and hence $\theta_1^2, \theta_1\theta_2, \theta_2^2$ are distinct (note that these characters are distinct if and only if their weights $(p_i,q_i)$ are distinct). This implies that $\eta_i$'s are distinct and $\pi'$ is regular. 
\end{proof}

\begin{prop}\label{not22}
    $\rho=\rho_{\pi,\fp}$ does not decompose into the direct sum of two $2$-dimensional representations. 
\end{prop}
\begin{proof}
By Proposition \ref{isInduction}, we have 
\[
\wedge^2 \pi = \mathrm{Ind}^{K'}_K(\pi')
\]
for a cuspidal automorphic representation $\pi'$ of $\GL_3(\AA_{K'})$. We also know that $K'$ is totally real by Lemma \ref{isTotReal} and $\pi'$ is regular algebraic by Lemma \ref{isRegular}. In particular, there is a compatible family of Galois representations attached to it and since it is cuspidal, by \cite[Theorem 1.2]{bockle2025weak}, these Galois representations are irreducible. This gives 
\[
\rho_{\pi',\fp} \oplus {}^\gamma \rho_{\pi',\fp} = \wedge^2 \rho_{\pi,\fp}|_{G_{K'}} = \det (\sigma)|_{G_{K'}} \oplus \det (\tau)|_{G_{K'}} \oplus (\sigma \otimes \tau)|_{G_{K'}},
\]
which is clearly a contradiction to the irreducibility of the 3-dimensional representation $\rho_{\pi',\fp}$. This completes the proof.    
\end{proof}
Putting everything together, we have shown Theorem A:
\begin{cor}\label{mainIrr}
    Let $K$ be a totally real field and $\pi$ be a regular algebraic cuspidal automorphic representation of $\GL_4(\AA_K)$ that is not essentially self-dual. Then for each embedding $\lambda:\QQ(\pi) \hookrightarrow \overline{\QQ_p}$, the Galois representation 
    \[
    \rho_{\pi,\lambda}:\Gamma_K \rightarrow \GL_4(\overline{\QQ_p})
    \]
    attached to $\pi$ is irreducible.
\end{cor}
\begin{proof}
    If $\pi$ is self-twisted, then there exists a degree two extension $L/K$ and a cuspidal automorphic representation $f$ of $\GL_2(\AA_L)$ such that
    \[
    \pi = \mathrm{Ind}^L_K(f).
    \]
Base changing to $L$, we have
\[
\mathrm{BC}^L_K(\pi) = f \boxplus {}^\gamma f,
\]
which implies that $f$ is regular algebraic. Since $L$ is either totally real or CM, this implies that there are Galois representations attached to $f$. Since $\pi$ is cuspidal, we must have $f \neq {}^\gamma f$, hence, the same is true for the associated Galois representations and therefore 
\[
\rho_{\pi,\lambda} = \mathrm{Ind}^L_K(\rho_{f,\lambda})
\]
is irreducible. If $\pi$ is not self-twisted, the statement follows from Proposition \ref{not31} and Proposition \ref{not22}, since $\rho_{\pi,\lambda}$ is semisimple by definition. 
\end{proof}

\subsection{Local-Global compatibility at $\ell=p$}
We keep the notations from the last section. In particular, $\pi$ is neither essentially self-dual nor self-twisted. Now that we have shown irreducibility of $\rho_{\pi,\lambda}$, an argument similar to \cite[Theorem 3.7]{bockle2025weak} can be used to show that these Galois representations satisfy the expected $p$-adic Hodge theoretic properties, at least for a density one set of primes. In fact, these properties are shown in \cite{a2024rigidity} under some mild conditions. We prove that these conditions are satisfied for a density one set of primes. 

Let us formulate this in a more general setting.
Only for the next theorem, $K$ could be a totally real or a CM field. 

\begin{theo}\label{localGlobal}
    Let $K$ be a totally real or CM field and $\Pi$ be a regular algebraic cuspidal automorphic representation of $\GL_n(\AA_K)$. Let $E \supseteq \QQ(\pi)$ be a number field and $\{\rho_{\Pi,\fp}:\Gamma_K\rightarrow \GL_n(\overline{E_\fp})\}_{\fp\in \Sigma_{E}}$ be the associated compatible family of Galois representations. Then there exists a density one set of rational primes $\mathcal{P}$ with the following property: 
    For every $p \in \mathcal{P}$ and $\fp|p$, if $\rho_{\Pi,\fp}$ is strongly irreducible, then:
    \begin{enumerate}
        \item $\rho_{\Pi,\fp}$ is de Rham with the expected Hodge-Tate weights.
        \item If $\Pi$ is unramified at $p$, then $\rho_{\Pi,\fp}$ is crystalline.
    \end{enumerate}    
\end{theo}
\begin{proof}
    In order to prove the claimed properties for $\rho_{\Pi,\fp}$, we use \cite[Theorem 1.0.6]{a2024rigidity}. Therefore, we only need to prove that $\overline{\rho_{\Pi,\fp}}$ is irreducible and decomposed generic. By \cite[Theorem 1.4]{bockle2026on}, there exists a density 1 set of primes $\mathcal{P}'$ such that above these primes $\rho_{\Pi,\fp}$ is irreducible if and only if $\overline{\rho_{\Pi,\fp}}$ is irreducible. This means that we only need to find a density 1 subset $\mathcal{P}$ of $\mathcal{P'}$ such that for $p \in \mathcal{P}$ and $\fp|p$, whenever $\overline{\rho_{\Pi,\fp}}$ is irreducible, it is decomposed generic.  
    \\
    The idea is to use the result of Larsen on hyperspecial maximal images in compatible families \cite[Theorem 3.17]{larsen1995maximality}. To use this, we need a compatible family with values in $\GL_n(\QQ_p)$ rather than $\GL_n(\overline{\QQ_p})$. Theorem \ref{FieldSplitsAll} implies that there exists a number field $E'$ that splits the algebraic monodromy group of $\rho_{\Pi,\fp}$ for all $\fp$. Thus, we can descend to a compatible system with values in $\GL_n(E'_{\fp'})$ for $\fp' \in \Sigma_{E'}$. Therefore, without loss of generality, we may assume that $E$ is large enough so that we have a compatible family
    \[
    \{\rho_{\Pi,\fp}:\Gamma_K \rightarrow \GL_n(E_\fp)\}_{\fp\in \Sigma_E}
    \]
    of Galois representations. 
    \\
    The rest of the proof is similar to the one presented in \cite[Theorem 3.7]{bockle2025weak} for the $\GL_3$ case. We include the argument for the sake of completeness. First of all, note that by \cite[Lemma 7.1.5]{allen2023potential}, we only need to show that $M_\fp$, the normal closure of $\overline{K}^{\ker(\mathrm{ad}(\overline{\rho_{\Pi,\fp}}))}$, does not contain a primitive $p$'th root of unity $\zeta_p$. 
    \\
    Now, we can construct a compatible family of Galois representations
    \[
    \rho_p=\prod_{\fp|p} \rho_{\Pi,\fp} :\Gamma_K \rightarrow \prod_{\fp|p}\GL_n(E_\fp) = \GL_n(E \otimes \QQ_p)=(\mathrm{Res}^E_\QQ\GL_n)(\QQ_p)\hookrightarrow \GL_{nd}(\QQ_p),
    \]
    where $d=[E:\QQ]$. Looking at the adjoint Galois representation gives the adjoint compatible family
    \[
    \rho^{\mathrm{ad}}_p=\prod_{\fp|p} \rho^{\mathrm{ad}}_{\Pi,\fp} :\Gamma_K \rightarrow \prod_{\fp|p}\GL_{n^2}(E_\fp) = \GL_{n^2}(E \otimes \QQ_p)\hookrightarrow \GL_{n^2d}(\QQ_p).
    \]
    Choose the finite Galois extension $L/K$ such that $\rho_p|_{\Gamma_L}$ has a connected monodromy group. Let $H_{p}^{\mathrm{ad}}/\QQ_p$ be the algebraic monodromy group of $\rho^{\mathrm{ad}}_p|_{\Gamma_L}$. Since we assumed that $\rho_p$ is strongly irreducible, $H_{p}^{\mathrm{ad}}/\QQ_p$ is in fact the adjoint quotient of the algebraic monodromy group of $\rho_p$. There exists a central isogeny $\alpha_p:H_p^{\mathrm{sc}} \rightarrow H_p^{\mathrm{ad}}$, where $H_p^{\mathrm{sc}}$ is a simply connected semisimple group. By \cite[Theorem 3.17]{larsen1995maximality}, there exists a density 1 set of primes $\mathcal{P}_1$ such that for any $p \in \mathcal{P}_1$, the preimage $\Gamma_p^{\mathrm{sc}}$ of $\Gamma_p^{\mathrm{ad}}:=\rho_p^{\mathrm{ad}}(\Gamma_L)$ in $H_p^{\mathrm{sc}}(\QQ_p)$ is a hyperspecial maximal compact subgroup. Now define
    \[
    \mathcal{P}_2 = \{p \in \mathcal{P}_1: [L(\zeta_p):L]=p-1 \}.
    \]
    The condition is clearly satisfied for large $p$, hence, $\mathcal{P}_2$ is of density 1. We claim that sufficiently large primes in $\mathcal{P}_2$ work for the theorem. 
    Let $p \in \mathcal{P}$
    and $\Omega_p$ be the normal closure of $L_1:=\overline{L}^{\ker(\mathrm{ad}({\rho_{\Pi,\fp}})|_{\Gamma_L})}$. Clearly, if we show that the (infinite) extension $\Omega_p/\QQ$ does not contain $\QQ(\zeta_p)$, we are done. 
    \\
    $L_1$ is Galois over $L$ but not necessarily over $\QQ$. It has at most $[L:\QQ]$ many $\Gamma_\QQ$-conjugates and $\Omega_p$ is the composite of all these conjugates. Let us denote them by 
    \[
    L_1,\dots,L_k.
    \]
    Then each $L_i$ is Galois over $L$ and
    \[
    \mathrm{Gal}(L_i/L)\simeq \Gamma_p^{\mathrm{ad}}.
    \]
    Since $\Omega_p$ is the composite of $L_i$'s, we have an embedding
    \[
    G_p:=\mathrm{Gal}(\Omega_p/\QQ) \hookrightarrow \prod_{i=1}^k \mathrm{Gal}(L_i/L) \simeq \prod_{i=1}^k \Gamma_p^{\mathrm{ad}}.
    \]
    Let us denote the residual image of $\Gamma_p^{\mathrm{ad}}=\rho_p^{\mathrm{ad}}(\Gamma_L)$ in $\GL_{n^2d} (\FF_p)$ by $\overline{\Gamma_p^{\mathrm{ad}}}$ (we can assume that $\rho^{\mathrm{ad}}_p$ has values in $\GL_{n^2d}(\ZZ_p)$ after conjugation). Let $\overline{G_p}$ be the image of $G_p$ in $\prod_1^k \overline{\Gamma_p^{\mathrm{ad}}}$. Since the kernel of $\GL_N(\ZZ_p) \rightarrow \GL_N(\FF_p)$ is a pro-$p$ group, \cite[Lemma 3.9]{bockle2025weak} implies that we have an inclusion of the multi-sets of Jordan-Hölder factors:
    \begin{equation}\label{JH}
        \mathrm{JH}\left(\frac{\ZZ}{(p-1)\ZZ}\right) \subset \mathrm{JH}\left(\prod_i\overline{\Gamma_p^{\mathrm{ad}}}\right)
    \end{equation}
    By \cite[Corollary 2.5]{hui2020maximality}, $\alpha_p(\Gamma_p^{\mathrm{sc}})\subset \Gamma_p^{\mathrm{ad}}$ is normal of finite index and 
    \[
    [\Gamma_p^{\mathrm{ad}}:\alpha_p(\Gamma_p^{\mathrm{sc}})]<C,
    \]
    for a constant $C$ only depending on $E$. Therefore, the same bound holds residually: 
    \[
    [\overline{\Gamma_p^{\mathrm{ad}}}:\overline{\alpha_p(\Gamma_p^{\mathrm{sc}})}]<C.
    \]
    Let $\mathcal{H}_p/\ZZ_p$ be an integral model for the hyperspecial group $\Gamma_p^{\mathrm{sc}}$. Hence 
    \[
    \mathcal{H}_p(\ZZ_p)\simeq\Gamma_p^{\mathrm{sc}}
    \]
    and the generic fiber of $\mathcal{H}_p/\ZZ_p$ is $H_p^{\mathrm{sc}}/\QQ_p$. 
    \\ For any finite group $G$, let $c_p(G)$ denote the product of the orders of the cyclic factors appearing in the multi-set $\mathrm{JH}(G)$ that are not isomorphic to $\ZZ/p\ZZ$. Since the kernel of
    \[
    \mathcal{H}_p(\ZZ_p) \rightarrow \mathcal{H}_p(\FF_p)
    \]
    is pro-$p$, $c_p(\mathcal{H}_p(\FF_p))$ is bounded by some $C'$ independent of $p$. Therefore, the same bound also holds residually. Combining all these together, we get
    \[
    c_p\left(\prod_{i=1}^k \overline{\Gamma_p^{\mathrm{ad}}}\right) \leq (CC')^k \leq (CC')^{[L:\QQ]}. 
    \]
    Now, if we choose $p>1+(CC')^{[L:\QQ]}$, this clearly contradicts the inclusion in (\ref{JH}). This implies the result for sufficiently large $p$ in $\mathcal{P}_2$ and we are done.
\end{proof}

\begin{cor}\label{p-adicHT}
    Let $K$ be a totally real field and $\pi$ be a regular algebraic cuspidal automorphic representation of $\GL_4(\AA_K)$. Then there exists a density one set of rational primes $\mathcal{P}$ with the following property: 
    For every $p \in \mathcal{P}$ and every embedding $\lambda:\QQ(\pi) \rightarrow \overline{\QQ_p}$ one has:
    \begin{enumerate}
        \item $\rho_{\pi,\lambda}$ is de Rham with the expected Hodge-Tate weights.
        \item If $\pi$ is unramified at $p$, then $\rho_{\pi,\lambda}$ is crystalline.
    \end{enumerate}    
\end{cor}
\begin{proof}
    If $\pi$ is essentially self-dual, it is polarizable by \cite[Theorem 2.1]{patrikis2015sign} and therefore $\rho_{\pi,\lambda}$ satisfies local-global compatibility by \cite{barnet2014local}. If $\pi$ is an induction of a degree two extension $K'/K$, then $K'$ is either totally real or CM and we are done by \cite[Theorem 7.1.10]{allen2023potential}. So, we may assume that $\pi$ is not essentially self-dual and not an induction. Therefore, $\rho_{\pi,\lambda}$ is irreducible by Corollary \ref{mainIrr}. Now,  \cite[Proposition 4.5]{calegari2013irreducibility} implies that either $\rho_{\pi,\lambda}$ is strongly irreducible, in which case Theorem \ref{localGlobal} implies the result, or $\rho_{\pi,\lambda}$ is potentially abelian. Then after a finite extension $L/K$ we have
    \[
    \rho_{\pi,\lambda}|_{\Gamma_L} = \chi_1 \oplus \chi_2 \oplus \chi_3 \oplus \chi_4.
\]
First, assume that not all four characters are equal. Then, we are in the situation of Lemma \ref{GerenalizedCG} with $k \neq 1$. This means that $\rho_{\pi,\lambda}$ and hence $\pi$ is an induction. Therefore, all the four characters must be equal. This implies that $\mathrm{ad}(\rho_{\pi,\lambda})$ has finite image and hence an Artin representation. In particular, the compatible system is independent of $\lambda$ (can be realized with coefficients in $\overline{\QQ}$). Let us assume that it factors through $\mathrm{Gal}(M/K)$ for a number field $M$ that can clearly be assumed to be Galois over $\QQ$. Then for every prime $p>1+[M:\QQ]$ and each $\lambda$ above $p$, $\rho_{\pi,\lambda}$ is decomposed generic by \cite[Lemma 7.1.5]{allen2023potential}. This implies that $\rho_{\pi,\lambda}$ is de Rham of the expected Hodge-Tate weights and in particular regular. This contradicts equality of the characters $\chi_1,\dots,\chi_4$.
\end{proof}

\section{Lie Algebra Classification}\label{LieClassification}
Let $K$ be a totally real field as usual and $\pi$ be a regular algebraic cuspidal automorphic representation of $\GL_4(\AA_K)$. In this section, we give an automorphic classification for the Lie algebras $\mathfrak{g}\subseteq \mathfrak{gl}_4$ (over $\overline{\QQ_p}$) that occur as the Lie algebra of the image of the $p$-adic Galois representations attached to $\pi$. Since the determinant of such a representation is given by the central character of $\pi$, we focus on classifying $\mathfrak{g}^{\mathrm{ss}}$. Almost everything in this section can be extracted from the literature (mostly from \cite{asgari2007cuspidality} on the automorphic side and from \cite{calegari2013irreducibility} on the Lie algebra side). Nevertheless, we think it is useful to gather everything in one place.\\
If $\pi$ is essentially self-dual, then $\wedge^2 \pi$ is not cuspidal by Theorem \ref{Asgari}. By the same proposition, $\pi$ is either an Asai transfer, an automorphic tensor product, or a generic transfer from $\mathrm{GSp_4}$. In the latter case, $\pi$ is either a symmetric cube transfer from $\GL_2$ or not. If it is not in the image of $\mathrm{sym}^3$, we say that $\pi$ is of \textbf{primitive $\mathrm{GSp}_4$ type} (or primitive symplectic type). We will see that in this case, the semisimple part of the Lie algebra of the image should be $\mathfrak{sp}_4$, which justifies the name (but we can only prove this over $\QQ$). \\
We say that $\pi$ is of \textbf{primitive $\mathrm{GL}_4$ type} (or simply primitive type) if $\pi$ is neither essentially self-dual nor self-twisted. We will see that in this case, the semisimple part of the Lie algebra of the image should be $\mathfrak{sl}_4$, which justifies the name.
If $\wedge^2 \pi$ is cuspidal, then by Theorem \ref{Asgari} $\pi$ is either an Asai transfer of a non-dihedral automorphic representation of $\GL_2$, or of primitive $\mathrm{GL}_4$ type. \\
Let $\mathfrak{g}_\lambda$ be the $\overline{\QQ_p}$-Lie algebra of $\rho_{\pi,\lambda}(\Gamma_K)\subset \GL_4(\overline{\QQ_p})$. If $\pi$ is an induction or an Asai transfer of a dihedral representation, then $\rho_{\pi,\lambda}$ is potentially abelian and $\mathfrak{g}_\lambda^{ss}=0$. If $\pi$ is an induction, an Asai transfer, or a symmetric cube of a non-dihedral representation, then this representation is regular algebraic (possibly after a twist) and $\mathfrak{g}_\lambda^{ss}$ is isomorphic to $\mathfrak{sl}_2$ by Lemma \ref{GL(2)}. In these cases, since everything reduces to $\GL_1$ and $\GL_2$, computing the image of the Galois representation is also straightforward. The non-trivial cases are when $\pi=\pi_1 \boxtimes \pi_2$ is an automorphic tensor, when $\pi$ is of primitive $\mathrm{GSp}_4$ type, and when it is of primitive $\GL_4$ type. 
\\
Let $\pi=\pi_1 \boxtimes \pi_2$ for cuspidal automorphic representations $\pi_1$ and $\pi_2$ of $\GL_2(\AA_K)$. If $\pi_1=\mathrm{Ind}^{K'}_K \chi$ for a Hecke character $\chi$ over a degree two extension $K'/K$, then
\[
\pi=\mathrm{Ind}^{K'}_K \chi \boxtimes\pi_2 = \mathrm{Ind}^{K'}_K (\chi\otimes \mathrm{BC}^{K'}_K \pi_2)
\]
and we are back to the induction case. So we assume that both $\pi_1$ and $\pi_2$ are non-dihedral.
This case has been worked out by Patrikis \cite[\S 2.6]{patrikis2019variations}. 
\begin{prop}\label{primitiveOrthogonal}
    Let $\pi$ be a RAC automorphic representation of $\GL_4(\AA_K)$. Assume $\pi=\pi_1 \boxtimes \pi_2$ for non-dihedral cuspidal automorphic representations $\pi_1$ and $\pi_2$ of $\GL_2(\AA_K)$. Then for every $\lambda$, one has $\mathfrak{g}_\lambda^{ss}=\mathfrak{so}_4 \subset \mathfrak{sl}_4$. 
\end{prop}
\begin{proof}
    By \cite[Proposition 2.5.8]{patrikis2019variations}, $\pi_1$ and $\pi_2$ can be assumed to be $W$-algebraic. Since $\pi$ is regular, it is a discrete series representation at the archimedean places. Therefore, $\pi_1$ and $\pi_2$ are Hilbert modular forms, possibly of mixed parity, whose parities are compatible. Then \cite[Proposition 2.6.7]{patrikis2019variations} implies that $\rho_{\pi,\lambda}$ is Lie-irreducible and we have a decomposition $\rho_{\pi,\lambda}=\rho_1\otimes \rho_2$ for two-dimensional representations $\rho_1$ and $\rho_2$. Note that $\pi_1$ and $\pi_2$ are not $L$-algebraic, so $\rho_1$ and $\rho_2$ are not the Galois representations attached to these automorphic representations (but they are so after a degree two base-change). \\
    Now both $\rho_1$ and $\rho_2$ have to be Lie-irreducible since $\rho_{\pi,\lambda}$ is so. Therefore, both have Lie algebra $\mathfrak{sl}_2$. This tells us that the Lie algebra of $\rho_1\otimes\rho_2$ must be $\mathfrak{so}_4\simeq \mathfrak{sl}_2 \times \mathfrak{sl}_2$. Note that the isomorphism is given by $\otimes$, so $\mathfrak{sl}_2 \times \mathfrak{sl}_2 \simeq \mathfrak{so}_4$ acts irreducibly. 
\end{proof}

\begin{prop}\label{primitiveSymplectic}
    Let $\pi$ be a RAC automorphic representation of $\GL_4(\AA_\QQ)$ that is of primitive $\mathrm{GSp}_4$ type. Then, for all but finitely many $\lambda$, one has $\mathfrak{g}_\lambda^{ss}=\mathfrak{sp}_4$.
\end{prop}
\begin{proof}
    $\pi$ is essentially self-dual and hence polarizable. 
    By \cite{hui2023monodromy}, $\rho_{\pi,\lambda}$ is irreducible for all but finitely many $\lambda$. Fix such a $\lambda$. Since $\pi$ is not self-twisted, it is not an induction from a quadratic extension and \cite[Corollary 4.4]{calegari2013irreducibility} implies that $\rho_{\pi,\lambda}$ is Lie-irreducible.\\
    By definition, $\pi$ comes from a globally generic representation of $\mathrm{GSp}_4$ and hence the Galois representation is symplectic. Now, \cite[Proposition 4.5]{calegari2013irreducibility} implies that the semisimple part $\mathfrak{g}_\lambda^{ss}$ of the Lie algebra of the image is either $\mathrm{sym}^3(\mathfrak{sl}_2)$ or $\mathfrak{sp}_4$. We only need to show that it cannot be $\mathrm{sym}^3(\mathfrak{sl}_2)$. If it is, then it has to be so for all $\lambda$, because $\mathfrak{sl}_2$ has rank 1 and the rank of the image is independent of $\lambda$ in a compatible family. Since $\mathrm{sym}^3$ is irreducible, Schur's Lemma implies that the whole Lie algebra $\mathfrak{g}_\lambda$ (not just the semisimple part) is contained in $\mathrm{sym}^3(\mathfrak{gl}_2)$. This implies that the connected component of the algebraic monodromy group $G_\lambda$ is contained in $\mathrm{sym}^3(\GL_2)$. Since the normalizer of $\mathrm{sym}^3(\GL_2)$ is trivial, the Galois representation factors via $\mathrm{sym}^3$. Therefore, $G_\lambda^{ss,\circ}=\mathrm{sym}^3(\SL_2)$. Then (for big enough $\lambda$) one can either use \cite[Proposition 3.11.7]{conti2016big} which uses the  Fontaine-Mazur conjecture or \cite[Theorem 1.2]{weiss2022images} which uses Serre's modularity conjecture to deduce that $\pi$ is the symmetric cube of an eigenform. This contradicts the assumption that $\pi$ is of primitive $\mathrm{GSp}_4$ type.
\end{proof}

\begin{rem}
    Note that the irreducibility results in \cite{hui2023monodromy} work over any totally real field. Therefore,
    if one has appropriate analogs of Serre's modularity conjecture or the Fontaine-Mazur conjecture over $K$, the above proof works over any totally real field. 
\end{rem}

\begin{prop}\label{nonself-dual}
    Let $\pi$ be a RAC automorphic representation of $\GL_4(\AA_K)$ that is of primitive $\mathrm{GL}_4$ type. Then, for every $\lambda$, one has $\mathfrak{g}_\lambda^{ss}=\mathfrak{sl}_4$.
\end{prop}
\begin{proof}
    The Lie algebra $\mathfrak{sl}_4$ is of type $A$, so if we show that the Lie algebra is $\mathfrak{sl}_4$ for one place $\lambda$, it is automatically $\mathfrak{sl}_4$ for all places. 
    By Corollary \ref{p-adicHT}, there exists $\lambda$ for which $\rho_{\pi,\lambda}$ is de Rham with regular Hodge-Tate weights. Since $\pi$ is not essentially self-dual, \cite[Proposition 4.5]{calegari2013irreducibility} implies that if $\rho_{\pi,\lambda}$ is not potentially abelian, then the Lie algebra is $\mathfrak{sl}_4$ and we are done. Assume that $\rho_{\pi,\lambda}$ is  potentially abelian. Then after a finite extension $L/K$ we have
    \[
\rho_{\pi,\lambda}|_{\Gamma_L} = \chi_1 \oplus \chi_2 \oplus \chi_3 \oplus \chi_4.
\]
Regularity of the Hodge-Tate weights implies that the four characters are distinct. So we are in the situation of Lemma \ref{GerenalizedCG} with $k \neq 1$. This means that $\rho_{\pi,\lambda}$ is an induction and in particular, it is self-twisted. This implies that $\pi$ is self-twisted which is a contradiction.
\end{proof}

Let us summarize what we have seen so far in a table. Compare the automorphic side with the classification of \cite[Theorem 1.1]{asgari2007cuspidality} and the Galois side with the table in \cite[Proposition 4.5]{calegari2013irreducibility}. In the following table, $K''$ is a degree 4 extension of $K$, $K'$ is a degree 2 extension of $K$, $\chi$ is a Hecke character of $\GL_1(\AA_{K''})$, $f$ is a non-dihedral cuspidal automorphic representation of $\GL_2(\AA_{K'})$, and $\pi_1$ and $\pi_2$ are non-dihedral cuspidal automorphic representations of $\GL_2(\AA_K)$.
\medskip
\begin{center}
\begin{tabular}{ |p{3.7cm}||p{3.7cm}|p{2.5cm}|p{2.5cm}|  }
 \hline
 \multicolumn{4}{|c|}{Classification of $\mathfrak{g}^{\mathrm{ss}}$} \\
 \hline
 automorphic side& type &Lie algebra &representation\\
 \hline
 1. $\pi = \mathrm{Ind}^{K''}_K (\chi)$   & self-twist    &$\mathfrak{g}^{ss}=0$&    $\mathrm{std}_4$\\
 2. $\pi = \mathrm{Ind}^{K'}_K (f)$&   self-twist  & $\mathfrak{g}^{ss}=\mathfrak{sl}_2$   &$\mathrm{std}_2\oplus \mathrm{std}_2$\\
 3. $\pi = \mathrm{As}^{K'}_K(f)$ &self-dual (orthogonal) & $\mathfrak{g}^{ss}=\mathfrak{sl}_2$&  $\mathrm{std}_2\otimes \mathrm{std}_2$\\
 4. $\pi = \pi_1 \boxtimes \pi_2$   &self-dual (orthogonal) & $\mathfrak{g}^{ss}= \mathfrak{so}_4$&  $\mathrm{std}_4$\\
 5. $\pi = \mathrm{sym}^3(\pi_1)$&   self-dual (symplectic)  & $\mathfrak{g}^{ss}=\mathfrak{sl}_2$&$\mathrm{sym}^3$\\
 6. $\pi$ : primitive $\mathrm{GSp}_4$& self-dual (symplectic) & $\mathfrak{g}^{ss}=\mathfrak{sp}_4$   &$\mathrm{std}_4$\\
 7. $\pi$ : primitive $\GL_4$& primitive  & $\mathfrak{g}^{ss}=\mathfrak{sl}_4$&$\mathrm{std}_4$\\
 \hline
\end{tabular}
\end{center}

\section{Extra-Twists for Reductive Groups}\label{SectionExtra}

In this section we develop the theory of extra-twists for reductive groups. This starts with Ribet's work on $\GL_2$-type abelian varieties where he introduces inner-twists. This is generalized by the author to the case of $\GL_n$, where one has both inner- and outer-twists \cite{shavali2025image}. Here, we generalize this to any split reductive group that admits a faithful irreducible representation and use it to study the image of the automorphic Galois representations in the next section. 

We will only define extra-twists on the Galois side since that is all we need here. That is why we denote our reductive group in which our representations land by $\widehat{G}$, to indicate that in practice this should be thought of as the Langlands dual group of some reductive group $G$. 

Let $\widehat{G}_{/\QQ}$ be a (connected) reductive group and let $A$ be a $\QQ_p$-algebra that is isomorphic to a finite product $\prod_i A_i$ where each $A_i$ is an algebraic extension of $\QQ_p$. Then each $A_i$ is equipped with a $p$-adic valuation and $A$ is equipped with the product topology. 
Now, let $K$ be a number field and $\Gamma_K$ its absolute Galois group. A $\widehat{G}$-valued $p$-adic Galois representation (of $\Gamma_K$) is a continuous homomorphism
\[
\rho:\Gamma_K \rightarrow \widehat{G}(A).
\]
Two especially important cases of this are when $A/\QQ_p$ is an algebraic field extension, or $A=E\otimes_{\QQ}\QQ_p=\prod_{\fp|p}E_\fp$ for a number field $E$. In the latter case, we say that $\rho$ has coefficients in $E$. 
\\
From now on, assume that $\widehat{G}_{/\QQ}$ is a split reductive group that admits a faithful irreducible representation $$\iota:\widehat{G}  \hookrightarrow \GL_n.$$
This implies that the center $\widehat{Z}$ of $\widehat{G}$ has a cyclic character group, and $\widehat{Z}_{\overline{\QQ}}$ is either finite cyclic or isomorphic to $\mathbb{G}_m$ by Lemma \ref{center}.
\begin{rem}
    The reason that we make this assumption on the existence of $\iota$ is that we need to work with traces of $\widehat{G}$-valued representations and it is easier to do this in $\GL_n$. We think that it is possible to drop this assumption if one is willing to work with Lafforgue's notion of pseudo-representations. Since we will only work with the groups $\GL_4, \mathrm{GSp}_4$, and $\mathrm{GSO_4}$ in the next chapter, we do not need this generality. 
\end{rem}
We have a commutative diagram of algebraic groups

\begin{center}
\begin{tikzcd}
                                      & \widehat{G}^{\mathrm{der}} \arrow[d] \arrow[rd]     &                           \\
\widehat{Z} \arrow[r] \arrow[rd, "r"] & \widehat{G} \arrow[r, "\mathrm{ad}"] \arrow[d, "v"] & \widehat{G}^{\mathrm{ad}} \\
                                      & \widehat{T}                                         &                          
\end{tikzcd}
\end{center}
where the row and the column are short exact sequences, the diagonal maps are isogenies, and $\widehat{T}$ is the maximal torus quotient of $\widehat{G}$. 

By a Galois character, we mean a $p$-adic Galois representation with values in $\widehat{Z}$, i.e. a continuous homomorphism
\[
\chi:\Gamma_K \rightarrow \widehat{Z}(A).
\]
If $\rho:\Gamma_K \rightarrow \widehat{G}(A)$ is a $p$-adic Galois representation, then $\rho \cdot \chi$ is also a $p$-adic Galois representation with values in $\widehat{G}(A)$. We denote this representation by $\rho \otimes \chi$ and call it the twist of $\rho$ with $\chi$. Note that $\widehat{Z}$ is either finite or one-dimensional. 

\begin{lemma}\label{kill_det}
    Let
    \[
    \rho:\Gamma_K \rightarrow \widehat{G}(A)
    \]
    be a $p$-adic Galois representation. Then there exists a finite Galois extension $L/K$ and a character $\chi$ of $\Gamma_L$ such that $\rho|_{\Gamma_L} \otimes \chi$ has values in $\widehat{G}^{\mathrm{der}}(A)$. 
\end{lemma}
\begin{proof}
    If $\widehat{Z}$ is finite then the statement is trivial, i.e. one just chooses $L$ sufficiently large to kill the image of $v$ and take $\chi$ to be trivial. So assume that $\widehat{Z}\simeq \mathbb{G}_m$. Therefore $\widehat{T} \simeq \mathbb{G}_m$ as well and the isogeny $r:\widehat{Z} \rightarrow \widehat{T}$ is given by $z\mapsto z^n$ for some integer $n$. Now the character $\mu:=v \circ \rho : \Gamma_K \rightarrow \widehat{T}(A)=A^\times$ has an $n$'th root after passing to a finite extension (because the $p$-adic logarithm would exist on an open neighborhood of the image) and we can go further up to the Galois closure $L$ of this extension. Then $\mu^{1/n}\otimes \rho|_{\Gamma_L}$ has clearly image in the kernel of $v$ and we are done.
\end{proof}

Now, assume that we have a $p$-adic Galois representation with coefficients in a number field $E$:
\[
\prod_{\fp|p}\rho_{\fp}=\rho_{p}: \Gamma_K \rightarrow \widehat{G}(E_p) = \prod_{\fp|p} \widehat{G}(E_\fp).
\]
Using the fixed representation $\iota:\widehat{G} \hookrightarrow \GL_n$, we can talk about the trace (or more generally the characteristic polynomial) of such representations. 
One can also use the embedding 
\[
E_p= E \otimes_{\QQ} \QQ_p \hookrightarrow E \otimes_{\QQ} \overline{\QQ_p}=:\overline{E_p}
\]
to consider this as a product of representations
\[
\prod_{\lambda:E \hookrightarrow \overline{\QQ_p}}\rho_{\lambda}=\rho_{p}: \Gamma_K \rightarrow \widehat{G}(\overline{E_p}) = \prod_{\lambda} \widehat{G}(\overline{E_\lambda})
\]
with coefficients in algebraically closed fields. We say that $\rho_p$ is absolutely irreducible if each $\rho_\lambda$ is irreducible (by which we mean it is irreducible as a $\GL_n(\overline{E_\lambda})$-valued representation). 

Now, we want to define extra-twists for $p$-adic Galois representations. First, note that an automorphism of $\widehat{G}$ maps $\widehat{Z}$ to $\widehat{Z}$ because it is surjective and hence the image of $\widehat{Z}$ commutes with everything. Since inner-automorphisms act trivially on $\widehat{Z}$, outer automorphisms act on Galois characters (with values in $\widehat{Z}$). Therefore, 
if $\eta:\Gamma_K \rightarrow \widehat{Z}(A) \subset \widehat{G}(A)$ is a character, then $\Phi\circ\eta$ is also a character which we denote by $\eta^\Phi$. 

\begin{defi}
    Let $E$ be a number field and $\rho_p$ a $\widehat{G}$-valued Galois representation with coefficients in $E$. An extra-twist of $\rho_p$ is a triple $(\sigma, \Phi, \chi)$ where $\sigma \in \mathrm{Aut}(E)$, $\Phi \in \mathrm{Out}(\widehat{G})$, and $\chi: \Gamma_K \rightarrow E_p^\times$ is a Galois character, such that
\[
\rho_p \simeq \Phi({}^\sigma \rho_p) \otimes \chi
\]
as $\GL_n$-valued Galois representations. An extra-twist $(\sigma,\Phi,\chi)$ is called an inner-twist if $\Phi$ is trivial. In this case we drop $\Phi$ from the notation. 
\end{defi}

We will see later that extra-twists form a group under the composition law:
\[
(\sigma, \Phi, \chi) \circ (\tau, \Psi, \eta) := (\sigma\tau, \Phi \circ \Psi, \chi\cdot {}^\sigma\eta^\Phi).
\]

To further investigate the properties of extra-twists, we will assume some further properties on our Galois representations. In particular, these assumptions will imply that $\rho_p$ has geometrically Zariski dense image in $\widehat{G}^{ss}$, i.e. the image of each 
\[
\rho_\lambda:\Gamma_K \rightarrow \widehat{G}(\overline{E_\lambda})
\]
is Zariski dense (modulo the center). In particular, the $\GL_n$-valued representation $\iota\rho_\lambda:\Gamma_K\rightarrow \widehat{G}(\overline{E_\lambda})\subseteq \GL_n(\overline{E_\lambda})$ is irreducible. This means that the isomorphism $\rho_p \simeq \Phi({}^\sigma \rho_p) \otimes \chi$ is equivalent to the equality of the traces of both sides, by the Brauer-Nesbitt Theorem. We will use this property often without explicitly mentioning it. 
\begin{defi}\label{valid}
We say that $\rho_p$ is valid if it satisfies the following properties:
\begin{enumerate}
    \item Each $\rho_\lambda$ is  unramified outside a finite set $S$ of places of $K$ containing the archimedean places and all places above $p$.
    \item $f^{(p)}_w(x):=\mathrm{CharPoly}(\rho_p(\mathrm{Fr}_w))$ has coefficients in $E \subset E\otimes_\QQ \QQ_p$, for each finite place $w \notin S$ of $K$. 
    \item $v(\rho_p)$ is trivial, i.e. $\rho_p$ has values in $\widehat{G}^{\mathrm{der}}$.
    \item For each embedding $\lambda:E\rightarrow \overline{\QQ_p}$, the $\overline{\QQ_p}$-Lie algebra of the $p$-adic Lie group $\rho_{\lambda}(\Gamma_K)$ is equal to $\widehat{\mathfrak{g}}^{\mathrm{der}}(\overline{\QQ_p}):=\mathrm{Lie}(\widehat{G}^{\mathrm{der}}_{\overline{\QQ_p}})$. In particular, each $\rho_\lambda$ is Lie-irreducible.
\end{enumerate}
\end{defi}
\begin{rem}\label{Zariski_dense}
    We make a few remarks about this definition.
    \begin{enumerate}
    \item The characteristic polynomial $\mathrm{CharPoly}(\rho_\lambda(\mathrm{Fr}_w))$ of $\rho_\lambda$ at $\mathrm{Fr}_w$ is given by applying the embedding $\lambda:E \rightarrow \overline{\QQ_p}$ to the polynomial  $f^{(p)}_w$. 
    \item Condition 3 is not as  restrictive as it may seem first, because by Lemma \ref{kill_det}, we can always trivialize $v$  after taking a finite extension of $K$ and a twist.
    \item Note that since $\hG$ is connected, condition (4) implies that the image of each $\rho_\lambda$ is $\overline{\QQ_p}$-Zariski dense in $\hGd$.
    \end{enumerate}
\end{rem}

\begin{lemma}\label{properties}
    Let $\rho_p$ be a valid $p$-adic Galois representation and $\Gamma$ be the set of all extra-twists of $\rho_p$. Then the following hold:
    \begin{enumerate}
        \item $\Gamma$ becomes a group under the composition law 
        \[
        (\sigma, \Phi, \chi) \circ (\tau, \Psi, \eta) := (\sigma\tau, \Phi \circ \Psi, \chi\cdot {}^\sigma\eta^\Phi).
        \]
        \item Every extra-twist $(\sigma,\Phi,\chi)$ is uniquely determined by $\sigma$. 
        \item One can identify $\Gamma$ with a subgroup of $\mathrm{Aut}(E)$ by forgetting $\Phi$ and $\chi$. In particular, $\Gamma$ is finite and if $F=E^\Gamma$ then $\Gamma=\mathrm{Gal}(E/F)$.
        \item For every extra-twist $(\sigma,\Phi,\chi)$, the character $\chi$ is finite. 
    \end{enumerate}
\end{lemma}
\begin{proof}
    For part (1), note that
    \[
    (\gamma,\Theta,\mu) \circ (\sigma\tau, \Phi \circ \Psi, \chi\cdot {}^\sigma\eta^\Phi) = (\gamma\sigma\tau, \Theta \circ \Phi\circ\Psi,\mu\cdot {}^\gamma \chi^\Theta \cdot {}^{\gamma\sigma}\eta^{\Theta\Phi}).
    \]
    On the other hand
    \[
    (\gamma\sigma, \Theta \circ \Phi, \mu \cdot {}^\gamma \chi^\Theta) \circ (\tau, \Psi, \eta) = (\gamma\sigma\tau, \Theta \circ \Phi\circ\Psi,\mu\cdot {}^\gamma \chi^\Theta \cdot {}^{\gamma\sigma}\eta^{\Theta\Phi}),
    \]
    which implies that composition is associative. It is clear that $(\mathrm{id}_E,\mathrm{id}_{\hG},1)$ is the identity element and it is easy to check that every element has an inverse. \\
    Let us prove (4) next. Assume that we have an extra-twist 
    \[
    \rho_p \simeq \Phi({}^\sigma \rho_p) \otimes \chi.
    \]
    Applying $v$ to both sides and using that $\rho_p$ has values in $\ker(v)= \hGd$ we obtain
    \[
    r(\chi) = 1.
    \]
    Since the center of $\hG$ embeds into $\GG_m$ by Lemma \ref{center}, the isogeny $r$ is given by $z \mapsto z^n$ for some $n$. This implies that $\chi^n=1$ and hence it is finite.
    \\
    For (2), assume that we have two extra-twists 
    \[
    (\sigma, \Phi, \chi) \: \mathrm{ and } \: (\sigma,\Psi, \eta).
    \]
    This implies that
    \[
    \Phi({}^\sigma \rho_p) \otimes \chi \simeq \Psi({}^\sigma \rho_p) \otimes \eta.
    \]
    Since $\chi$ and $\eta$ are finite, after passing to a finite extension $L/K$ to trivialize them and then applying $\sigma^{-1}$ we have
    \[
    \rho_p|_{\Gamma_L} \simeq \Phi^{-1}\Psi(\rho_p|_{\Gamma_L}).
    \]
    Since $\Phi$ and $\Psi$ are only defined up to conjugation, we can lift them to actual automorphisms of $\hG$ such that 
    \[
    \rho_p|_{\Gamma_L} = \Phi^{-1}\Psi(\rho_p|_{\Gamma_L}).
    \]
    If $\Phi \neq \Psi$, this is an algebraic equation that cuts down a proper subgroup of $\hGd$. This contradicts part 3 of Remark \ref{Zariski_dense}. Therefore, we must have $\Phi = \Psi$. Now, this implies that
    \[
    \Phi({}^\sigma \rho_p) \otimes \chi\eta^{-1} \simeq \Phi({}^\sigma \rho_p),
    \]
    and if $\chi \neq \eta$, it means that $\rho_p$ is self-twisted which again contradicts Remark \ref{Zariski_dense}. \\
    Part (3) is an immediate consequence of part (2).
\end{proof}

\section{Big Image Results}\label{SectionBig}
In this section, we prove our main big image results. The main technical tool is the description of the field fixed by all the extra-twists. 
\subsection{Computing the Lie algebra}
We keep the notations from the last section. So $p$ is a prime number, $K$ an arbitrary number field, and $\widehat{G}/\QQ$ is a split reductive group that admits a faithful irreducible representation.
Let 
$$\rho_p : \Gamma_K \rightarrow \widehat{G}(E \otimes \QQ_p)$$ 
be a valid ($\widehat{G}$-valued) Galois representation with coefficients in a number field $E$, as in Definition \ref{valid}. Let $\Gamma \subset \mathrm{Aut}(E)$ be the group of the extra-twists of $\rho_p$ and $E=F^\Gamma$  be the field fixed by all the extra-twists of $\rho_p$. 

\begin{lemma}\label{charTwist}
    Let $L$ be a finite Galois extension of $K$ and $\rho_1,\rho_2:\Gamma_K \rightarrow \widehat{G}(E_p)$ absolutely irreducible $p$-adic Galois representations such that $\rho_1|_{\Gamma_L}\simeq \rho_2|_{\Gamma_L}$. Then there exists a ($\widehat{Z}$-valued) Galois character $\chi$ such that
    \[
    \rho_1 \simeq \rho_2 \otimes \chi.
    \]
\end{lemma}
\begin{proof}
    Conjugating $\rho_1$, we may assume that $\rho_1$ and $\rho_2$
    are in fact equal when restricted to $\Gamma_L$. Now define
    $$
    \phi (g) := \rho_{1}^{-1}(g)\rho_{2}(g).
    $$
    A priori $\phi$ is just a map
    $\phi:\Gamma_K \rightarrow \widehat{G}(E_p)$ which
    is trivial on $\Gamma_L$. We want to prove that it
    is actually a homomorphism with values
    in the center (hence a character) on all of $\Gamma_K$. 

    Let $g\in \Gamma_K$ and $h \in \Gamma_L$. 
    Note that
    $\rho_{1} (h) = \rho_{2} (h)$
    and 
     $\rho_{1} (ghg^{-1}) = \rho_{2} (ghg^{-1})$
     since $\Gamma_L$ is normal in $\Gamma_K$.
     Now the following computation shows that 
     $\phi (g)=\rho_{1}^{-1}(g)\rho_{2}(g)$
     commutes with $\rho_{2} (h)$:
     $$
     \rho_{1}^{-1}(g)\rho_{2}(g)\rho_{2} (h)
     = \rho_{1}^{-1}(g)\rho_{2}(gh) =
     \rho_{1}^{-1}(g)
     \rho_{2}(ghg^{-1})\rho_{2}(g)
     $$
     $$
     =\rho_{1}(g^{-1})
     \rho_{1}(ghg^{-1})\rho_{2}(g)
     =\rho_{1} (h)\rho_{1}^{-1}(g)\rho_{2}(g)
     = \rho_{2} (h)\rho_{1}^{-1}(g)\rho_{2}(g)
     $$
     Now since $V_2$ is absolutely irreducible when restricted to $\Gamma_L$, we have
     $\mathrm{End}_{\Gamma_L}(V_2)=\overline{\QQ_{p}}$, hence $\phi(g)$ is scalar and it commutes with everything. By its definition, $\phi(g)$ has values in $\widehat{G}$ so it must land in the center of $\widehat{G}$. 
\end{proof}

\begin{lemma}\label{cent}
    Let $\Gamma_L = \bigcap_{\sigma} \ker\chi_\sigma$ and $L'/L$ a finite extension that is Galois over $K$. For every outer automorphism $\Phi$ of $\widehat{G}$ and every unramified finite place $v$ of $L'$, let $a_{v,\Phi}=\mathrm{tr}(\Phi(\rho_p(\mathrm{Fr}_v)))$. Then 
    $
    F = \QQ \left( \left\{\sum_{\Phi \in \mathrm{Out}(\widehat{G})} a_{v,\Phi}\right\}_v \right).
    $
\end{lemma}
\begin{proof}
    We first show that $
    F'=\QQ \left( \left\{\sum_{\Phi} a_{v,\Phi}\right\}_v \right) \subseteq F
    $. For each $\sigma \in \Gamma$, we have an isomorphism 
    \[
    \rho_p|_{\Gamma_{L'}} \simeq {}^\sigma \Phi_\sigma (\rho_p|_{\Gamma_{L'}}).
    \]
    Applying an outer automorphism $\Psi$ of $\widehat{G}$ to both sides and looking at the trace of  $\mathrm{Fr}_v$ we get $a_{v,\Psi}=\sigma (a_{v,\Psi\Phi_\sigma})$, therefore
    \[
    \sum_{\Psi \in \mathrm{Out}(\widehat{G})} a_{v,\Psi} = \sigma \left( \sum_{\Psi \in \mathrm{Out}(\widehat{G})} a_{v,\Psi\Phi_\sigma} \right).
    \]
    Since $\Phi_\sigma$ is fixed, $\Psi\Phi_\sigma$ runs over all outer automorphisms and hence
    \[
    \sum_{\Psi \in \mathrm{Out}(\widehat{G})} a_{v,\Psi} = \sigma \left( \sum_{\Psi \in \mathrm{Out}(\widehat{G})} a_{v,\Psi\Phi_\sigma} \right)=\sigma \left( \sum_{\Psi \in \mathrm{Out}(\widehat{G})} a_{v,\Psi} \right)
    \]
    This means that $\sum_{\Phi} a_{v,\Phi}$
    is invariant under the action of $\Gamma$ which implies $
    F' \subseteq F
    $. 
    \\ To prove the other inclusion, we need to show that every $\tau \in \mathrm{Gal}(\overline{\QQ}/F')$ induces an extra-twist of $\rho_p$. Consider the semisimple Galois representation
    \[
    \tilde{\rho}=\left(\oplus_{\Phi \in \mathrm{Out}(\widehat{G})} \Phi(\rho_p)|_{\Gamma_{L'}} \right).
    \]
    The trace of this representation at $\mathrm{Fr}_v$ is $\sum_{\Phi} a_{v,\Phi}$. Since Frobenius elements are dense, if $\tau$ fixes all this elements, then it fixes the trace of $\tilde{\rho}$, and by the Brauer-Nesbitt Theorem
    \[
    {}^\tau \tilde{\rho} \simeq \tilde{\rho}.
    \]
    Since each summand is irreducible we get
    \[
    {}^\tau \Phi(\rho_p) |_{\Gamma_{L'}}\simeq \rho_p|_{\Gamma_{L'}} 
    \]
    for some $\Phi$. Now, by Lemma \ref{charTwist} there exists a Galois character $\chi$ such that 
    \[
    \Phi({}^\tau \rho_p) \otimes \chi\simeq \rho_p.
    \]
    This implies that $(\tau|_E, \Phi,\chi)$ is an extra-twist. Therefore $F'$ is exactly the field fixed by all extra-twists. 
\end{proof}

\begin{lemma}\label{isoRep}
    Let $\Gamma_L = \bigcap_{\sigma \in \Gamma} \ker\chi_\sigma$ and $L'/L$ a finite extension that is Galois over $K$.
    Assume that $\lambda,\mu: E \hookrightarrow \overline{\QQ_p}$ are two embeddings and $\Phi$ an outer automorphism of $\widehat{G}$. If $\rho_\lambda|_{\Gamma_{L'}} \simeq \Phi(\rho_\mu)|_{\Gamma_{L'}}$ (this mean conjugation in $\GL_n$) then $\lambda|_F=\mu|_F$.
\end{lemma}
\begin{proof}
    Assume that $\rho_\lambda \simeq \Phi(\rho_\mu)$ as representations of $\Gamma_{L'}$. Applying $\Psi \in \mathrm{Out}(\widehat{G})$ to both sides and looking at the trace of $\mathrm{Fr}_v$ for an unramified finite place $v$ of $L'$ we have $\lambda(a_{v,\Psi})=\mu(a_{v,\Psi\Phi})$. Therefore
    \[
    \lambda \left( \sum_{\Psi \in \mathrm{Out}(\widehat{G})} a_{v,\Psi} \right) = \mu \left( \sum_{\Psi \in \mathrm{Out}(\widehat{G})} a_{v,\Psi\Phi} \right) = \mu \left( \sum_{\Psi \in \mathrm{Out}(\widehat{G})} a_{v,\Psi} \right),
    \]
    which implies that $\lambda$ and $\mu$ agree on $F$ by Lemma \ref{cent}.
\end{proof}
Now we can compute the $\QQ_p$-Lie algebra of the image of $\rho_p$.
First we construct an algebraic group from the extra-twists.
\\
Recall that since we assumed $\rho_p$ is valid, $v(\rho_p)$ is trivial. Therefore we have
$$
\rho_p : \Gamma_K \rightarrow \hGd (E_p).
$$

We first define a 1-cocycle $f:\Gamma \rightarrow \mathrm{Aut}_{E_p}(\hGd)$ using extra-twists. Let 
\[
\Gamma_L = \bigcap_{\sigma\in \Gamma} \chi_\sigma
\]
as usual. Since all $\chi_\sigma$'s are finite by Lemma \ref{properties}, $L/K$ is a finite extension.
For every extra-twist $(\sigma,\Phi_\sigma,\chi_\sigma) \in \Gamma$, we note that 
\[
\mathrm{tr}\left( \rho_p|_{\Gamma_L}\right) = \mathrm{tr}\left(\Phi({}^\sigma\rho_p|_{\Gamma_L})\right).
\]
Since $\rho_p$ is strongly irreducible, these two representation are ($\GL_n$-)isomorphic  and there exists $\alpha_\sigma \in \GL_n(E_p)$ such that $$\rho_p|_{\Gamma_{L}} = \alpha_\sigma \cdot \Phi({}^\sigma \rho_p|_{\Gamma_{L}}) \cdot \alpha^{-1}_\sigma.$$
Note that by Remark \ref{Zariski_dense} (3), the Zariski closure of the image of $\rho_p|_{\Gamma_{L}}$ is $\hGd$ so $\mathrm{ad}(\alpha_\sigma)$ in stabilizes $\hGd$. Now define $f:\Gamma \rightarrow \mathrm{Aut}_{E_p}(\hGd)$ by
\[
\sigma \mapsto \mathrm{ad}(\alpha_\sigma) \circ \Phi.
\]
\begin{lemma}
    The function $f$ defined above is a 1-cocycle in $Z^1(\Gamma,\mathrm{Aut}_{E_p}(\hGd))$.
\end{lemma}
\begin{proof}
    Let us check the cocycle condition for $(\sigma,\Phi,\chi)$ and $(\tau,\Psi,\eta)$:
$$
f_{\tau}\cdot{}^\tau f_\sigma= (\mathrm{ad}(\alpha_\tau)\circ \Psi) \circ (\mathrm{ad}({}^\tau \alpha_\sigma)\circ \Phi) = \mathrm{ad}(\alpha_\tau \cdot \Psi({}^\tau \alpha_\sigma))\circ \Psi{}\Phi.
$$
On the other hand applying $\tau$ and then $\Psi$ to  $$\rho_p|_{\Gamma_{L}} = \alpha_\sigma \cdot \Phi({}^\sigma \rho_p|_{\Gamma_{L}}) \cdot \alpha^{-1}_\sigma$$ we get
$$
\alpha_\tau^{-1} \cdot \rho_p|_{\Gamma_{L}} \cdot \alpha_\tau = \Psi({}^\tau \alpha_\sigma) \cdot \Psi\Phi({}^{\tau \sigma}\rho_p|_{\Gamma_{L}})\cdot \Psi({}^\tau \alpha_\sigma^{-1}),
$$
which implies that $$f_{\tau \sigma} = \mathrm{ad}(\alpha_\tau \cdot \Psi({}^\tau \alpha_\sigma))\circ \Psi\Phi$$
and we are done. 
\end{proof}

Now, similar to \cite[\S 2.2]{shavali2025image}, we can define the twisted action of $\Gamma$ by this cocycle on $\hGd$. From the construction of $f$, it is clear that every matrix in the image of $\Gamma_L$ is invariant under this twisted action. Let $H$ be the algebraic group 
$$H \coloneqq (\mathrm{Res}^{E_p}_{F_p} \hGd)^{tw_f(\Gamma)}$$  
defined over $F_p$. Then it follows:
\begin{cor}\label{groupH}
     The representation $\rho_p|_{\Gamma_L}$ factors through $H(F_p)\subseteq \hGd(E_p)$, i.e.  $\rho_p(\Gamma_L)\subseteq H(F_p)$. 
\end{cor}
Note that by Proposition \cite[Proposition 2.6]{shavali2025image}, $H$ is a form of $\hGd$ and in particular is a semisimple group. Also, note that since $\Gamma_L$ is open in $\Gamma_K$, the Lie algebra of the $p$-adic Lie groups $\rho_p (\Gamma_K)$ and $\rho_p (\Gamma_L)$ are the same. Let $\mathfrak{g}$ be the Lie algebra (over $\QQ_p$) of $\rho_p (\Gamma_L)$ and $\mathfrak{h}$ be the Lie algebra of the algebraic group $H/F_p$, both viewed as Lie-subalgebras of $\mathrm{Lie}_{\QQ_p}(\hGd(E_p))$. Our next goal is to show that these two Lie algebras are in fact equal. 

\begin{prop}\label{LieAlgImage}
    With the notation as above, $\mathfrak{g}=\mathfrak{h}$.
\end{prop}
\begin{proof}
    First note that $\mathfrak{g} \subseteq \mathfrak{h}$ by the last corollary. Since $\mathfrak{h}$ is semisimple, it suffices to prove that $\overline{\mathfrak{g}^{\mathrm{der}}} = \mathfrak{g}^{\mathrm{der}} \otimes_{\QQ_p} \overline{\QQ_p}$ is equal to $\overline{\mathfrak{h}} = \mathfrak{h} \otimes_{\QQ_p} \overline{\QQ_p}$.
\\
    For every embedding $\sigma : F \hookrightarrow \overline{\QQ_p}$, fix an extension $\widetilde{\sigma} : E \hookrightarrow \overline{\QQ_p}$ of $\sigma$. All of the other extensions of $\sigma$ can be obtained by composing with different elements of the Galois group $\Gamma=\mathrm{Gal}(E/F)$, i.e. are of the form $\widetilde{\sigma}\tau$ for $\tau \in \Gamma$. Now we base change our representation $\rho_p|_{\Gamma_L}$
    to $\overline{\QQ_p}$ to get: 
\begin{center}
\begin{tikzcd}
\rho_p:\Gamma_L \arrow[r] \arrow[rd] & \hGd(E_p) \arrow[r, hook]        &  \hGd(\overline{E_p}) \arrow[r, "="]       & \mathrm{Res}^E_F(\hGd)(\overline{F_p})                              \\
                                & H(F_p) \arrow[r, hook] \arrow[u, hook] & H(\overline{F_p}) \arrow[r, "="] \arrow[u, hook] & \mathrm{Res}^E_F(\hGd)^{tw(\Gamma)}(\overline{F_p}) \arrow[u, hook]
\end{tikzcd}
\end{center}
\bigbreak
 \noindent where $\overline{E_p}:=E_p \otimes_{\QQ_p} \overline{\QQ_p}$ and
  $\overline{F_p}:=F_p \otimes_{\QQ_p} \overline{\QQ_p} = \prod_{\sigma : F \hookrightarrow \overline{\QQ_p}}\overline{\QQ_p}$. Note that we have
    
    $$
    \mathrm{Res}^E_F(\hGd)(\overline{F_p}) \simeq
    \prod_{\sigma:F \hookrightarrow \overline{\QQ_p}} \hGd (E \otimes_{F,\sigma}\overline{\QQ_p}) \simeq
    \prod_{\sigma:F \hookrightarrow \overline{\QQ_p}} \hGd (E \otimes_F E \otimes_{E,\widetilde{\sigma}}\overline{\QQ_p}) 
    $$
    $$
    \simeq 
    \prod_{\sigma:F \hookrightarrow \overline{\QQ_p}} \prod_{\Gamma} \hGd(\overline{\QQ_p}) \simeq
    \prod_{\lambda:E \hookrightarrow \overline{\QQ_p}} \hGd(\overline{\QQ_p}),
    $$
\bigbreak
   \noindent where $\lambda = \widetilde{\sigma}\tau$ for $\tau \in \Gamma$. By \cite[Proposition 2.6]{shavali2025image}, projection to the id-component of $\Gamma$ gives the isomorphism of the form $\mathrm{Res}^E_F(\hGd)^{tw(\Gamma)}$ of $\hGd$ with $\hGd$ over $E_p$. Therefore, we have: 
\begin{center}
\begin{tikzcd}
\rho_p:\Gamma_L \arrow[r] \arrow[rd] & \hGd(E_p) \arrow[r, hook]        & \prod_{\sigma:F \hookrightarrow \overline{\QQ_p}} \prod_{\Gamma} \hGd(\overline{\QQ_p}) \arrow[r, "="] & \prod_{\lambda:E \hookrightarrow \overline{\QQ_p}} \hGd(\overline{\QQ_p}) \\
                                & H(F_p) \arrow[r, hook] \arrow[u, hook] & \prod_{\sigma:F \hookrightarrow \overline{\QQ_p}}  \hGd(\overline{\QQ_p}) \arrow[u, hook]              &                                                                                   
\end{tikzcd}
\end{center}
\bigbreak
\noindent  For each embedding $\sigma : F \hookrightarrow \overline{\QQ_p}$
    the composition
    $$
    \rho_{\sigma}: \Gamma_{L} \rightarrow
    H(F_p) \hookrightarrow 
    \prod_{\sigma : F \hookrightarrow \overline{\QQ_p}} \hGd(\overline{\QQ_p})
    \xrightarrow{\mathrm{pr}_\sigma}
    \hGd(\overline{\QQ_p})
    $$
    corresponded to the representations $\rho_\lambda$ of $\Gamma_L$ for some embedding $\lambda:E \hookrightarrow \overline{\QQ_p}$ extending $\sigma$. Note that by Lemma \ref{isoRep}, these $\rho_\sigma$'s does not become isomorphic even after applying any outer automorphism and after any finite extension. This is the main point of the rest of the argument. 
\\
    On the level of the Lie algebras, this gives the embedding
    $$\overline{\mathfrak{g}^{\mathrm{der}}} \subseteq \overline{\mathfrak{h}} \xrightarrow{\simeq}  \prod_{\sigma : F \hookrightarrow \overline{\QQ_p}} \widehat{\mathfrak{g}}^{\mathrm{der}}(\overline{\QQ_p}).$$
    Let
    $\mathfrak{g}^{{\mathrm{der}}}_{\sigma} \subseteq \widehat{\mathfrak{g}}^{\mathrm{der}}(\overline{\QQ_p})$  be the projection of $\overline{\mathfrak{g}^{\mathrm{der}}}$ to the $\sigma$-component of the above map. This is the $\overline{\QQ_p}$-Lie algebra of the image of the representation $\rho_\sigma$ (=$\rho_\lambda$ for some $\lambda$ extending $\sigma$), so by our assumption on $\rho_p$ being valid, we have $\mathfrak{g}^{{\mathrm{der}}}_{\sigma} = \widehat{\mathfrak{g}}^{\mathrm{der}}(\overline{\QQ_p})$. 
\\
    Now we can apply \cite[Lemma 4.6]{ribet1980twists} to
    $$\overline{\mathfrak{g}^{\mathrm{der}}} \subseteq \overline{\mathfrak{h}} \rightarrow  \prod_{\sigma : F \hookrightarrow \overline{\QQ_p}} \widehat{\mathfrak{g}}^{\mathrm{der}}(\overline{\QQ_p}).$$
    We only need to prove that for every $\sigma, \tau : F \hookrightarrow \overline{\QQ_p}$ the projections 
    $(\mathrm{pr}_\sigma \times \mathrm{pr}_\tau) (\overline{\mathfrak{g}^{\mathrm{der}}})$ and $(\mathrm{pr}_\sigma \times \mathrm{pr}_\tau) (\overline{\mathfrak{h}})$ are equal. We follow the arguments of \cite[\S6.2]{serre1972proprietes}.
    
    Clearly it is enough to show that
    $(\mathrm{pr}_\sigma \times \mathrm{pr}_\tau) (\overline{\mathfrak{g}^{\mathrm{der}}}) = \widehat{\mathfrak{g}}^{\mathrm{der}}(\overline{\QQ_p}) \times \widehat{\mathfrak{g}}^{\mathrm{der}}(\overline{\QQ_p})$. Note that the first factor corresponds to the representation $\rho_{\sigma}$
    and the second to $\rho_{\tau}$. By the Lie algebra version of the Goursat's theorem \cite[Lemma 5.2.1]{ribet1976galois}, if $(\mathrm{pr}_\sigma \times \mathrm{pr}_\tau) (\overline{\mathfrak{g}^{\mathrm{der}}})$ is not equal to $\widehat{\mathfrak{g}}^{\mathrm{der}}(\overline{\QQ_p}) \times \widehat{\mathfrak{g}}^{\mathrm{der}}(\overline{\QQ_p})$
    then it has to be the graph of an isomorphism. Let us call this automorphism
    $\phi:\widehat{\mathfrak{g}}^{\mathrm{der}}(\overline{\QQ_p}) \rightarrow \widehat{\mathfrak{g}}^{\mathrm{der}}(\overline{\QQ_p})$. Since $\widehat{\mathfrak{g}}^{\mathrm{der}}$ is simple, the group of its outer automorphisms is the group of automorphisms of $\hG$ and hence $\phi$ is in the class of some $\Phi \in \mathrm{Out}(\hG)$. Fixing a lift $\tilde{\Phi} \in \mathrm{Aut}(G)$, the automorphism $\phi$ only differ by it up to conjugation. 
    In other words, we have the following diagram:

\begin{center} 
\begin{tikzcd}
                                     &  & \widehat{\mathfrak{g}}^{\mathrm{der}}(\overline{\QQ_p}) \arrow[dd, "\phi=ad(\alpha)\circ \tilde{\Phi}"] \\
\overline{\mathfrak{g}} \arrow[rru, "pr_{\sigma}"] \arrow[rrd, "pr_{\tau}"] &  &                                                 \\
                                     &  & \widehat{\mathfrak{g}}^{\mathrm{der}}(\overline{\QQ_p})                               
\end{tikzcd}
\end{center}  
which means that $\rho_{\sigma}$ and $\Phi(\rho_{\tau})$ are isomorphic as representations of $\overline{\mathfrak{g}}$. This implies that they are indeed isomorphic as representations of some open subgroup of $\Gamma_L$, which contradicts Lemma \ref{isoRep}. This implies the result. 
\end{proof}

Now, applying \cite[ Proposition 18.17]{schneider2011p} to the identity map $\mathfrak{g} \rightarrow \mathfrak{h}$ coming from Proposition \ref{LieAlgImage} gives the following result.

\begin{cor}\label{bigImCor}
    The image of $\Gamma_L$ under $\rho_p$ is an open subgroup of the $p$-adic Lie group $H(F_p)$.
\end{cor}

\subsection{Application to automorphic Galois representations}
We can now apply the results of the previous section to the Galois representations attached to automorphic representations of $\GL_4$ and prove Theorem B. Let $K$ be a totally real field and $\pi$ be a RAC automorphic representation of $\GL_4(\AA_K)$ as usual. As we saw in Chapter \ref{LieClassification}, the three cases for which determining the image does not reduce to $\GL_1$ and $\GL_2$ are the primitive orthogonal case (when $\pi = \pi_1 \boxtimes \pi_2$ for non-dihedral $\pi_1$ and $\pi_2$), the primitive symplectic ($\mathrm{GSp}_4$) case, and the primitive ($\GL_4$) case. In the primitive symplectic case we will assume that $K=\QQ$. 

Let us first define extra-twists for automorphic representations. We will not actually need this, since we always work on the Galois side. Nevertheless, it would be mentally helpful to have in mind that these extra-twists come from the extra-twists on the automorphic side where one can probably compute them more easily. 
\begin{defi}
    Let $G/\QQ$ be a split reductive group and $\pi$ be an $L$-arithmetic automorphic representation of $G(\AA_K)$ in the sense of Definition \cite[3.1.3]{buzzard2014conjectural}. Let $E \supset \QQ(\pi)$ be a number field. An $E$-extra-twist of $\pi$ is a triple $(\sigma, \Phi,\chi)$ where $\sigma \in \mathrm{Aut}(E)$, $\Phi \in \mathrm{Out}(G)$, and $\chi$ is a character of the center of $G$, such that
    \[
    \pi \approx \Phi({}^\sigma \pi) \otimes \chi ,
    \]
    by which we mean that for all but finitely many unramified places $v$ of $K$, the Satake parameters of both sides agree at $v$.
\end{defi}
\begin{rem}
    Note that, for $G=\GL_n$, this actually implies 
    \[
    \pi = \Phi({}^\sigma \pi) \otimes \chi ,
    \]
    by strong multiplicity 1. This might not be true in general. Since we will always assume that $G$ admits an irreducible faithful representation and we will always transfer to $\GL_n$ to compare Galois representations by their traces, this does not matter for us. In general, if one wants to work with automorphic representations that are only cuspidal for $G$ but the transfer to $\GL_n$ is not cuspidal, then it might be better to define the extra-twists for $L$-packets rather than individual automorphic representations. 
\end{rem}

\subsubsection{Non-self-dual case}
Assume that $\pi$ is of primitive $\GL_4$ type, i.e. it is neither essentially self-dual, nor self-twisted. Let $\{\rho_{\pi,p}=\prod_{\fp|p}\rho_{\pi,\fp}\}_p$ be the compatible family of Galois representations attached to it. 
Then by Proposition \ref{nonself-dual}, the semisimple part of the Lie algebra of the image of each $\rho_{\pi,\lambda}$ is equal to $\mathfrak{sl}_4$. Therefore, by Lemma \ref{kill_det}, after a finite extension and a twist, $\rho_{\pi,\lambda}$ becomes valid in the sense of Definition \ref{valid}. Therefore, the image of $\rho_{\pi,p}$ is described by Corollary \ref{bigImCor}. 

Let us be more explicit. Every inner-twist of $\pi$ is obviously of the form $\pi={}^\sigma\pi \otimes \chi$. On the other hand, $\GL_4$ has one non-trivial outer automorphism given by $A \mapsto A^{-T}$. Therefore, an outer-twist of $\pi$ would be of the form $\pi = {}^\sigma\pi^\vee \otimes \chi$. As explained in the previous section, these extra-twists form a group $\Gamma$ and induce a twisted action (given by a 1-cocycle $f$) using which we defined 
$$H_p \coloneqq (\mathrm{Res}^{E_p}_{F_p} \mathrm{SL}_4)^{tw_f(\Gamma)}$$
where $F=E^\Gamma$. The group $H_p$ is a form of $\mathrm{SL}_4$ that splits over $E_p$. If the central character of $\pi$ is finite, then after a finite extension the image of $\rho_{\pi,p}$ is in $H_p(F_p)$ and is open in the $p$-adic topology. Otherwise, after a finite extension the image is in $\mathrm{Res}^{F_p}_{\QQ_p}H\cdot \mathbb{G}_{m,\QQ_p}$ and is open $p$-adically. In any case we deduce:
\begin{theo}\label{mainnonself-dual}
    Let $K$ be a totally real field and $\pi$ be a RAC automorphic representation of $\GL_4(\AA_k)$ that is not self-twisted. Assume that $\pi$ is not essentially self-dual. Then for every prime $p$, the (connected) semisimple part of the $\QQ_p$-Zariski closure of the image of $\rho_{\pi,p}$ is of the form $\mathrm{Res}^{F_p}_{\QQ_p} H_p$ where $H_p/F_p$ is a form of $\mathrm{SL}_4$. 
\end{theo}
Let $\Fi/F$ be the field fixed by the inner-twists of $\pi$. Since $\SL_4$ only has one outer-twist, $[\Fi:F]=1$ or $2$. The group $H_p$ becomes an inner form of $\SL_4$ after base change to $\Fi\otimes \QQ_p$. Therefore, if $F=\Fi$ or $\Fi \otimes \QQ_p=F_p \times F_p$, then $H_p$ is an inner-form of $\SL_4$. Otherwise it is an inner form of $\mathrm{SU}_4$. \\
A well-known result of Larsen implies that the image of the Galois representation in the universal cover of the adjoint of the monodromy group is hyperspecial maximal compact for a density one set of primes. 
Since $H_p$ is a form of $\SL_4$, it is simply connected and therefore the image in $(\mathrm{Res}^{F_p}_{\QQ_p}H_p)(\QQ_p)$ is hyperspecial maximal compact. In particular, for such primes $p$, the group $\mathrm{Res}^{F_p}_{\QQ_p}H_p$ is unramified, and in particular quasi-split. Therefore, every fiber of $H_p$ is quasi-split (note that $H_p$ is defined over $\mathrm{Spec}(F_p)= \bigsqcup_{\fp|p}\mathrm{Spec}(F_\fp)$). This implies that for a density 1 set of primes $\mathcal{P}$, each fiber of $H$ is either $\SL_4$ or $\mathrm{SU}_4$, and the image in each fiber is a hyperspecial maximal compact. Therefore, in each fiber the residual image either contains $\SL_4(\FF_p)$ or $\mathrm{SU}_4(\FF_p)$. In particular, if $H_p$ is an inner-form (for instance if $F^{\mathrm{inn}}_p=F_p \times F_p$) and $p \in \mathcal{P}$, then the residual image of $\rho_{\pi,\lambda}$ contains $\SL_4(\FF_p)$. 

We say that $p$ splits completely in the extension $\Fi/F$ when $\Fi \otimes \QQ_p = F_p \times F_p$. Let us summarize the discussion above:
\begin{theo}\label{residualBigImage}
    Let $K$ be totally real and $\pi$ be a RAC automorphic representation of $\GL_4(\AA_K)$ that is of primitive type. Then there exists a density 1 set of primes $\mathcal{P}$ such that for every primes $p \in \mathcal{P}$ and each finite place $\lambda$ of $\QQ(\pi)$ above $p$, the image of $\overline{\rho_{\pi,\lambda}}$ contains either $\SL_4(\FF_p)$ or $\mathrm{SU}_4(\FF_p)$. Moreover, if $p$ splits completely in $\Fi/F$ then the image of $\overline{\rho_{\pi,\lambda}}$ contains $\SL_4(\FF_p)$. 
\end{theo}
\begin{rem}\label{residualRemark}
    If one know the local-global compatibility of $\rho_{\pi,\lambda}$ at $\ell=p$, then the arguments in \cite{hui2016type} for type A representations (instead of \cite{larsen1995maximality}) show that one can improve Theorem \ref{residualBigImage} from density 1 to all but finitely many primes, and even prove an adelic big image theorem. 
\end{rem}
\begin{rem}
    If one knows explicit examples of primitive automorphic representations of $\GL_4$, the above result could imply new cases of the inverse Galois problem for $\SL_4(\FF_p)$. As far as we know,  there is no explicit example of such automorphic representations available in the literature. The reason seems to be that the minimal conductor of such an example needs to be highly ramified at small primes, so the computations become tedious. See \cite{ash2008cohomology} for a discussion on this.   
\end{rem}
\begin{exmp}
    Even though we do not know an explicit example of a primitive automorphic representation for $\GL_4$, there is an explicit example of a 4-dimensional compatible family of Galois representations that should correspond to such an automorphic representation \cite{scholten1999non}. One can directly apply our big image result to this family. In fact, this family has been studied by Dieulefait and Vila. In this case, Theorem \ref{residualBigImage} (together with Remark \ref{residualRemark}) recovers \cite[Theorem 7.3]{dieulefait2008geometric}. 
\end{exmp}

\subsubsection{Orthogonal case}
Let  $\pi = \pi_1 \boxtimes \pi_2$ for non-dihedral cuspidal automorphic representations $\pi_1$ and $\pi_2$ and let $\{\rho_{\pi,p}\}_p$ be the compatible family of Galois representations attached to it as usual. 
In other words, $\pi$ comes from the transfer from the split $G=\mathrm{GSpin}_4$ to $\GL_4$. Therefore, the Galois representation attached to it has values in the Langlands dual $\hG=\mathrm{GSO_4}$:
\[
\rho_{\pi,p}:\Gamma_K \rightarrow \mathrm{GSO_4}(E_p) \subset \GL_4(E_p).
\]
By Proposition \ref{primitiveOrthogonal}, the semisimple part of the  Lie algebra of the image of each $\rho_{\pi,\lambda}$ is $\mathfrak{so}_4$ which is equal to $\mathrm{Lie}_{\overline{\QQ}_p}(\rho_{\pi,\lambda}(\Gamma_K))^{ss}$. Therefore, by Lemma \ref{kill_det}, after a finite extension and a twist, $\rho_{\pi,\lambda}$ becomes valid in the sense of Definition \ref{valid} for the group $\hG=\mathrm{GSO_4}$. Therefore, the image of $\rho_{\pi,p}$ is described by Corollary \ref{bigImCor}. 

Let us be more explicit. It is easy to check that an inner-twist $(\sigma,\chi)$ of $\pi$ consists of two inner-twists
$$\pi_1={}^\sigma\pi_1\otimes \chi_1$$
and
$$\pi_2={}^\sigma\pi_2\otimes \chi_2$$
of $\pi_1$ and $\pi_2$, such that $\chi_1 \chi_2 = \chi$. On the other hand, the outer automorphism of $\pi$ switches $\pi_1$ and $\pi_2$. Therefore, an outer-twist $(\sigma,(\cdot)^{-T},\chi)$ is given by 
$$\pi_1={}^\sigma\pi_2\otimes \chi_1$$
and
$$\pi_2={}^\sigma\pi_1\otimes \chi_2,$$
such that $\chi_1 \chi_2 = \chi$. Note that one immediately sees that in this case $(\sigma^2,\chi\cdot{}^\sigma\chi)$ is an inner-twist. As explained in the previous section, these extra-twists form a group $\Gamma$ and induce a twisted action (given by a 1-cocycle $f$) using which we define 
$$H_p \coloneqq (\mathrm{Res}^{E_p}_{F_p} \mathrm{SO}_4)^{tw_f(\Gamma)}$$
where $F=E^\Gamma$. The group $H_p$ is a form of $\mathrm{SO}_4$ that splits over $E_p$. If the central character of $\pi$ is finite, then after a finite extension the image of $\rho_{\pi,p}$ is in $H_p(F_p)$ and is open in the $p$-adic topology. Otherwise, after a finite extension the image is in $\mathrm{Res}^{F_p}_{\QQ_p}H_p\cdot \mathbb{G}_{m,\QQ_p}$ and is open. In any case we deduce:
\begin{theo}\label{mainOrthogonal}
    Let $K$ be a totally real field and $\pi$ an orthogonal RAC automorphic representation of $\GL_4(\AA_k)$ that is not self-twisted. Assume that $\pi$ is not an Asai transfer. Then for every prime $p$, the (connected) semisimple part of the $\QQ_p$-Zariski closure of the image of $\rho_{\pi,p}$ is of the form $\mathrm{Res}^{F_p}_{\QQ_p} H_p$ where $H_p/F_p$ is a form of $\mathrm{SO}_4$. 
\end{theo}
\begin{rem}
    One can compare our result in this case with \cite[Appendix C]{sweeting2025bloch}, where the image in $\GL_2 \times \GL_2$ is studied. The condition assumed in \cite[Theorem C.3.6]{sweeting2025bloch} is exactly to exclude outer twists. 
\end{rem}

\subsubsection{Symplectic case}
Assume that $\pi$ is symplectic and $K=\QQ$. In other words, $\pi$ comes from the transfer from the group $G=\mathrm{GSp}_4$ to $\GL_4$. Therefore, the Galois representation attached to it has values in the Langlands dual $\hG=\mathrm{GSp_4}$:
\[
\rho_{\pi,p}:\Gamma_\QQ \rightarrow \mathrm{GSp_4}(E_p) \subset \GL_4(E_p).
\]
By Proposition \ref{primitiveSymplectic}, the semisimple part of the  Lie algebra of the image of each $\rho_{\pi,\lambda}$ is $\mathfrak{sp}_4$ for all but finitely many $\lambda$. Therefore, by Lemma \ref{kill_det}, after a finite extension and a twist, $\rho_{\pi,\lambda}$ becomes valid in the sense of Definition \ref{valid} for the group $\hG=\mathrm{GSp_4}$. Therefore, the image of $\rho_{\pi,p}$ is described by Corollary \ref{bigImCor}. 

The group $\mathrm{GSp}_4$ has no non-trivial outer automorphisms, so $\pi$ only has inner-twists in this case. These inner-twists form a group $\Gamma$ and induce a twisted action (given by a 1-cocycle $f$) using which we define 
$$H_p \coloneqq (\mathrm{Res}^{E_p}_{F_p} \mathrm{Sp}_4)^{tw_f(\Gamma)}$$
where $F=E^\Gamma$. The group $H_p$ is an inner form of $\mathrm{Sp}_4$ that splits over $E_p$. If the central character of $\pi$ is finite, then after a finite extension the image of $\rho_{\pi,p}$ is in $H_p(F_p)$ and is open in the $p$-adic topology. Otherwise, after a finite extension the image is in $\mathrm{Res}^{F_p}_{\QQ_p}H_p\cdot \mathbb{G}_{m,\QQ_p}$ and is open again. In any case we deduce:
\begin{theo}\label{mainSymplectic}
    Let $\pi$ be a symplectic RAC automorphic representation of $\GL_4(\AA_\QQ)$ that is not self-twisted. Assume that $\pi$ is not a symmetric cube transfer. Then for all but finitely many primes $p$, the (connected) semisimple part of the $\QQ_p$-Zariski closure of the image of $\rho_{\pi,p}$ is of the form $\mathrm{Res}^{F_p}_{\QQ_p} H_p$ where $H_p/F_p$ is a form of $\mathrm{Sp}_4$. 
\end{theo}

\begin{rem}
    Similar to the non-self-dual case, \cite[Theorem 3.17]{larsen1995maximality} easily implies that for a density one set of primes, the image of the residual representation $\overline{\rho_{\pi,\lambda}}$ contains $\mathrm{Sp}_4(\FF_p)$. This recovers \cite[Theorem 1.2 (ii)]{weiss2022images} for a density one set of primes. 
\end{rem}

Finally we deduce Theorem B from the above.
\begin{cor}\label{mainBigIm}
    Let $K$ be a totally real field and $\pi$ be a RAC automorphic representation of $\GL_4(\AA_K)$. If $\pi$ is symplectic then assume $K=\QQ$. Suppose that $\pi$ is not self-twisted, not an Asai transfer, and not a symmetric cube. Let $G/\QQ_p$ be the $p$-adic arithmetic monodromy group of the Galois representation $\rho_{\pi,p}$ attached to $\pi$. Then for all but finitely many primes, $G^{\mathrm{der}}$ is of the form
    \[
    G^{\mathrm{der}} = \mathrm{Res}^{F_p}_{\QQ_p}(H_p) 
    \]
    where $F$ is the field fixed by the extra-twists and $H_p$ is a form of $\mathrm{Sp}_4$, $\mathrm{SO}_4$, or $\SL_4$ if $\pi$ is symplectic, orthogonal, or neither of them, respectively. Moreover, if $\pi$ is not symplectic, then this holds for all primes.
\end{cor}
\begin{proof}
    This is a combination of Theorems \ref{mainnonself-dual}, \ref{mainOrthogonal}, and \ref{mainSymplectic}.
\end{proof}

\bibliographystyle{amsplain}
\bibliography{Final}

\end{document}